\documentclass[oneside,english,12pt]{amsart}
\usepackage{hyperref}
\usepackage{txfonts}
\usepackage{amsfonts}
\usepackage{lmodern}

\usepackage[T1]{fontenc}
\usepackage[latin9]{inputenc}
\usepackage{geometry}
\usepackage{amsthm}
\usepackage{amssymb}
\usepackage[all]{xy}
\usepackage{pinlabel}
\usepackage{amssymb,amsmath,amsthm,tikz}

\usetikzlibrary{matrix,arrows,decorations.pathmorphing}

\title{Exact Lagrangian Fillings of Legendrian $(2,n)$ torus links}
\author{Yu Pan}

\makeatletter
\numberwithin{equation}{section}
\numberwithin{figure}{section}

\theoremstyle{plain}
\newtheorem{thm}{Theorem}[section]
\newtheorem{lem}[thm]{Lemma}
\newtheorem{cor}[thm]{Corollary}
\newtheorem{prop}[thm]{Proposition}

\theoremstyle{definition}
\newtheorem{defn}[thm]{Definition}
\newtheorem{eg}[thm]{Example}

\theoremstyle{remark}
\newtheorem{rmk}[thm]{Remark}

\newcommand{\bbR}{{\mathbb{R}}}
\newcommand{\bbZ}{{\mathbb{Z}}}

\newcommand{\bbF}{{\mathbb{F}}}

\newcommand{\calA}{{\mathcal{A}}}

\newcommand{\calM}{{\mathcal{M}}}

\newcommand{\diag}[1]{\begin{center}\begin{minipage}{5in}\xymatrix{#1}\end{minipage}\end{center}}
\def\arl{\ar[l]}
\def\arr{\ar[r]}
\def\ard{\ar[d]}

\usepackage{babel}

\tikzset{node distance=1.5cm, auto}

\begin{document}

\maketitle

\abstract 
For a Legendrian $(2,n)$ torus knot or link with maximal Thurston-Bennequin number,
Ekholm, Honda, and K{\'a}lm{\'a}n \cite{EHK} constructed $C_n$ exact Lagrangian fillings, where $C_n$ is the $n$-th Catalan number.
We show that these exact Lagrangian fillings are pairwise non-isotopic through exact Lagrangian isotopy.
To do that, we compute the augmentations induced by the exact Lagrangian fillings $L$ to $\bbZ_2[H_1(L)]$ and distinguish the resulting augmentations.
\endabstract

\section{Introduction}

A  Legendrian submanifold $\Lambda$ in the standard contact manifold $\big(\bbR^3, \xi= \ker \alpha \big)$, where $\alpha = dz-y dx$, is a $1$-dimensional closed manifold such that  $T\Lambda \subset \xi$ everywhere. 
An  exact Lagrangian filling $L$ of $\Lambda$ in the symplectization manifold $\big(\bbR_t \times \bbR^3, \omega= d(e^t \alpha) \big)$ is a $2$-dimensional surface that is cylindrical over $\Lambda$ when $t$ is sufficiently large. 
See Section \ref{DGAmap} for the detailed definition and Figure \ref{filling} for a picture.
\begin{figure}[iht]
\begin{minipage}{2.5in}
\begin{center}
\labellist
\pinlabel $\Lambda$ at -10 120
\pinlabel $L$ at -5 55 
\endlabellist
\includegraphics[width=1in]{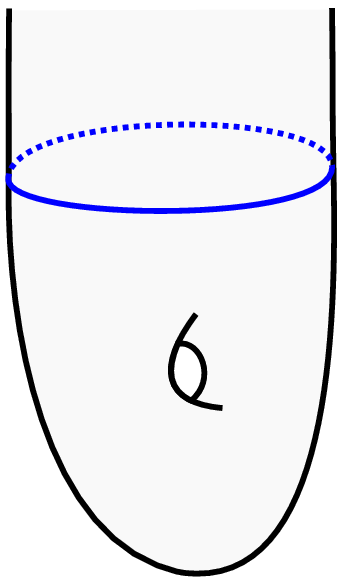}
\caption{A picture of an exact Lagrangian filling.}
\label{filling}
\end{center}
\end{minipage}
\begin{minipage}{3.7in}
\begin{center}
\labellist
\pinlabel  $b_1$ at 18 48
\pinlabel  $b_2$ at 45 48
\pinlabel $b_n$ at  250 48
\pinlabel $a_1$ at 275 115
\pinlabel $a_2$ at 275 20
\endlabellist
\includegraphics[width=2.8in]{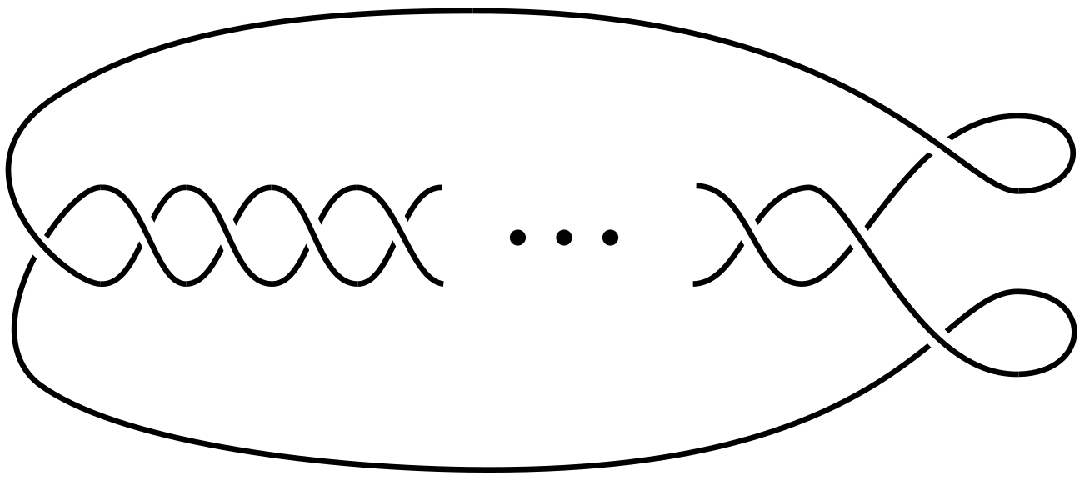}
\caption{The Lagrangian projection of the Legendrian $(2,n)$ torus knot.
The {\bf Lagrangian projection} is a projection from $\bbR^3$ to the $xy$-plane.}
\label{2n}
\end{center}
\end{minipage}
\end{figure}

In this paper, we study oriented exact Lagrangian fillings of the Legendrian $(2,n)$  torus links $\Lambda$ with maximal Thurston-Bennequin number ($n>0$).
When $n$ is even, we also require the link to have the right Maslov potential such that  Reeb chords $b_1, \dots, b_n$ in Figure \ref{2n} are in degree $0$ (see Section \ref{DGA} for detailed definitions). 
Ekholm, Honda, and K{\'a}lm{\'a}n \cite{EHK} gave an algorithm to construct exact Lagrangian fillings of the Legendrian $(2,n)$ torus link $\Lambda$ as follows.
Starting with a Lagrangian projection of $\Lambda$ as shown in Figure \ref{2n}, we can successively resolve  crossings $b_i$ in any order through  pinch moves (see Figure \ref{cob}), which  correspond to  saddle cobordisms.
As a result, we get two Legendrian unknots, which admit minimum cobordisms as shown in Figure \ref{cob}.
Concatenating the $n$ saddle cobordisms with these two minimum cobordisms, we get an exact Lagrangian filling of $\Lambda$.
\begin{figure}[!ht]
\labellist
\pinlabel  {The pinch move} at 30 -10
\pinlabel {The minimum cobordism} at 170 -10
\large
\pinlabel $\emptyset$ at 153 20
\endlabellist

\includegraphics[width=2.5in]{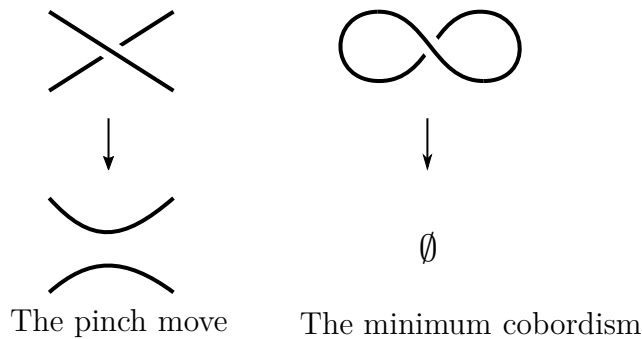}

\vspace{0.2in}

\caption{The pinch move and the minimum cobordism between Lagrangian projections of links.}
\label{cob}
\end{figure}

Different orders of resolving  crossings $b_1, \dots, b_n$ may give different exact Lagrangian fillings of $\Lambda$ up to exact Lagrangian isotopy.
Given a permutation $\sigma= \big(\sigma(1), \sigma(2), \dots, \sigma(n)\big)$ of $\{1, \dots, n\}$, write  $L_{\sigma}$ for the exact Lagrangian filling achieved by
 doing $n$ successive pinch moves at $b_{\sigma(1)}$, $b_{\sigma(2)}$, \dots , $b_{\sigma(n)}$, respectively, and then concatenating with the two minimum cobordisms.
Observe that two permutations may give isotopic exact Lagrangian fillings.
For instance, let $\Lambda$ be the Legendrian $(2,3)$ torus knot and consider the exact Lagrangian fillings of $\Lambda$ that correspond to permutations $(1,3,2)$ and $(3,1,2)$, respectively.
Since the saddles corresponding to the pinch moves at $b_1$ and $b_3$ are disjoint when projected to $\bbR^3$, one can use a Hamiltonian vector field in the $t$ direction to exchange the heights of these two saddles.
Therefore, the two fillings
$L_{(1,3,2)}$ and $L_{(3,1,2)}$ are Hamiltonian isotopic and thus are exact Lagrangian isotopic.
In general, for the Legendrian $(2,n)$ torus link $\Lambda$, given any numbers $i,j,k$ such that $i<k<j$, two permutations 
$(\dots, i, j, \dots, k, \dots)$ and $(\dots, j, i, \dots, k, \dots)$, where only $i$ and $j$ are interchanged, give the same exact Lagrangian fillings of $\Lambda$ up to exact Lagrangian isotopy.
Taking all the permutations of $\{1, \dots, n\}$ modded out by this relation, we obtain $C_n$
exact Lagrangian fillings of $\Lambda$,
where
$$C_n=\displaystyle{\frac{1}{n+1} {{2n}\choose{n}}}$$
is the $n$-th Catalan number.
In this paper, we prove the following theorem.
\begin{thm}[see Theorem \ref{knot} and Corollary \ref{link}]\label{main}
The $C_n$ exact Lagrangian fillings that come from the algorithm in  \cite{EHK}  are all of different exact Lagrangian isotopy classes.
In other words, the Legendrian $(2,n)$ torus link has at least $C_n$ exact Lagrangian fillings up to exact Lagrangian isotopy.
\end{thm}
Shende, Treumann, Williams and Zaslow  \cite{STWZ} have also constructed $C_n$ exact Lagrangian fillings of the Legendrian $(2,n)$ torus knot using cluster varieties  and shown that they are distinct  up to Hamiltonian isotopy. 
They remarked that these are presumably the same as fillings obtained by \cite{EHK}.
But we do not resolve this issue here.

\begin{rmk}
We will see from Corollary \ref{link} that
the conclusion of Theorem \ref{main}  for the case $n$ even  can be derived from the result for the case when $n$ is odd. 
Therefore,  for most of the paper, we focus on the case when $n$ is odd, which means $\Lambda$ is  a knot.
\end{rmk}

Inspired by \cite{EHK}, we use augmentations to distinguish the $C_n$ exact Lagrangian fillings of 
the Legendrian $(2,n)$ torus knot $\Lambda$.
In order to talk about augmentations, we first introduce
the Chekanov-Eliashberg differential graded algebra (DGA) of a Legendrian knot $\Lambda$, which is a chain complex $(\calA(\Lambda), \partial)$. 
This is an invariant of Legendrian submanifolds introduced by Chekanov \cite{Che} and Eliashberg  \cite{Eli}  in the spirit of symplectic field theory \cite{EGH}.
The underlying algebra $\calA(\Lambda)$ of the Chekanov-Eliashberg DGA is freely generated by Reeb chords of $\Lambda$ over a commutative ring $\bbZ_2[H_1(\Lambda)]=\bbZ_2[s, s^{-1}]$, where
  Reeb chords of $\Lambda$  correspond to double points of the Lagrangian projection of $\Lambda$.
The differential is defined by a count of rigid holomorphic disks with boundary on $\Lambda$, taken with coefficients in $\bbZ_2[H_1(\Lambda)]$.
In general, the Chekanov-Eliashberg DGA of $\Lambda$ is defined with $\bbZ[H_1(\Lambda)]$ coefficients. 
For our purpose, it  suffices to consider the DGA with $\bbZ_2[H_1(\Lambda)]$ coefficients, which means ignoring the orientations of moduli spaces of holomorphic disks.
An {\bf augmentation} $\epsilon$ of $\calA(\Lambda)$ to a commutative ring $\bbF$ is a DGA map  $\epsilon: (\calA(\Lambda), \partial) \to (\bbF, 0)$.
As shown in \cite{EHK}, an exact Lagrangian filling $L$ of $\Lambda$ gives an augmentation of $\calA(\Lambda)$ by counting rigid holomorphic disks with boundary on $L$.
Moreover, by \cite[Theorem 1.3]{EHK}, exact Lagrangian isotopic fillings give homotopic augmentations.
Therefore, in order to distinguish two fillings, we only need to show their induced augmentations are not chain homotopic.

In \cite{EHK}, the authors distinguished all the exact Lagrangian fillings from the algorithm when $n=3$ by computing all the augmentations of the Legendrian $(2,3)$ torus knot  to $\bbZ_2$ and finding 
that they are pairwise non-chain homotopic. 
However, when $n\ge5$, a computation shows that the number of augmentations of the  DGA  to $\bbZ_2$  is much less than the Catalan number $C_n$.

In this paper, for an exact Lagrangian filling $L$ of the Legendrian $(2,n)$ torus knot $\Lambda$, 
we consider its induced augmentation of $\calA(\Lambda)$ to $\bbZ_2[H_1(L)]$, where $H_1(L)$ is the singular homology of $L$.
Note that $H_1(L) \cong H_2(\bbR\times \bbR^3, L)$ and thus it is natural to count the rigid holomorphic disks in $\bbR\times \bbR^3$ with boundary on $L$ with $\bbZ_2[H_1(L)]$ coefficients.
However, the computation of augmentations is not as easy as the case with $\bbZ_2$ coefficients.
For each exact Lagrangian filling $L$ from the \cite{EHK} algorithm, 
we give a combinatorial formula of the induced augmentation of $\calA(\Lambda)$ to $\bbZ_2[H_1(L)]$.
Observing from the formula, we find a combinatorial invariant to show that the augmentations are pairwise non-chain homotopic.
In this way, we distinguish all the $C_n$ exact Lagrangian fillings of the Legendrian $(2,n)$ torus knot $\Lambda$ up to exact Lagrangian isotopy.
\vspace{.2in}

{\bf Outline.}
In Section \ref{Prelim}, we  review the Chekanov-Eliashberg DGA of a Legendrian submanifold and the DGA maps induced by an exact Lagrangian cobordism. 
In Section \ref{results}, we compute all the augmentations of the Legendrian $(2,n)$ torus knot to $\bbZ_2[H_1(L)]$ induced by the exact Lagrangian fillings $L$ and prove  that all the resulting augmentations are distinct up to chain homotopy.
In the end, we prove Theorem \ref{main} for the case $n$ even as a corollary.
\vspace{.2in}

{\bf Acknowledgement.}
The author would like to thank Lenhard Ng for introducing the problem and many enlightening discussions. This work was partially supported by NSF  grants
DMS-0846346 and DMS-1406371.

\section{Preliminaries} \label{Prelim}
In Section \ref{DGA}, we review the definition of Chekanov-Eliashberg DGA of  Legendrian submanifolds in $(\bbR^3, \ker \alpha)$ and  its extension to the setting of multiple base points.
For the purpose of computing augmentations in Section \ref{aug}, 
the definition of DGA we use here is slightly different from the versions in \cite{NgRSFT} and \cite{NRSSZ}, where  the underlying algebra  is completely non-commutative.
In our definition, we allow elements in the coefficient ring to commute with the elements corresponding to Reeb chords. 
This is a generalization of the definition of Chekanov-Eliashberg DGA from \cite{ENS}.
See \cite[Section 2.3.2]{EENS} for further discussions.
In Section \ref{DGAmap}, we review the DGA map induced by an exact Lagrangian cobordism
and revise coefficients of this map for the purpose of computing augmentations in Section \ref{aug}.

\subsection{Chekanov-Eliashberg DGA}\label{DGA}

Let $\Lambda$ be a Legendrian submanifold in $(\bbR^3, \ker \alpha)$, where $\alpha=dz-ydx$.
There are two projection diagrams associated to $\Lambda$ via the {\bf Lagrangian projection} $\Pi_{xy}: \bbR^3 \to \bbR^2, \ (x,y,z) \mapsto (x, y)$ and the  {\bf front projection} $\Pi_{xz}: \bbR^3 \to \bbR^2, \ (x,y,z) \mapsto (x, z)$ respectively.
As an example, 
 a front projection and a Lagrangian projection of the Legendrian trefoil are shown in Figure \ref{trefoil}.
Moreover, starting from a front projection of $\Lambda$,  Ng \cite{Ngresolve} gave an algorithm to get a Lagrangian projection of $\Lambda$ by smoothing the cusps of the front projection in a way shown in Figure \ref{cusp}.

\begin{figure}[!ht]
\includegraphics[width=3in]{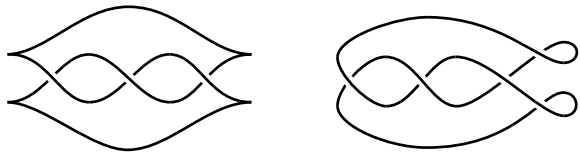}\\
\caption{A front projection (left) and a Lagrangian projection (right) of the Legendrian trefoil.}
\label{trefoil}
\end{figure}

\begin{figure}[!ht]
\includegraphics[width=3in]{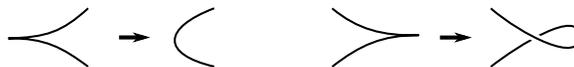}
\caption{Ng's algorithm to transfer a front projection to a Lagrangian projection by smoothing the left cusp directly and smoothing the right cusp with an additional crossing.}
\label{cusp}
\end{figure}

Let $\Lambda=\Lambda_1 \cup \Lambda_2 \cup \cdots \cup \Lambda_k$ be an oriented Legendrian link with $k$ connected components. 
Now let us define the Chekanov-Eliashberg DGA $\big(\calA(\Lambda; \bbZ_2[H_1(\Lambda)] ), \partial\big)$ of $\Lambda$.
To simplify the definition of  grading, we assume throughout the paper that the rotation number of $\Lambda$ is 0.
Note that all the Legendrian $(2,n)$ torus links we consider have rotation number $0$.

{\bf The underlying algebra.}
The underlying algebra $\calA(\Lambda; \bbZ_2[H_1(\Lambda)] )$ is a unital graded algebra freely generated by Reeb chords of $\Lambda$ over $\bbZ_2[H_1(\Lambda)] =\bbZ_2[s^{\pm 1}_1, s^{\pm 1}_2, \dots, s^{\pm 1}_k]$, 
where $\{s_1, s_2, \dots, s_k\}$ is any basis of $H_1(\Lambda)$.
A {\bf Reeb chord} of $\Lambda$ in $\big(\bbR^3, \ker \alpha\big)$ is a vertical line segment ($z$ direction) with both ends on $\Lambda$  endowed with an orientation in the positive $z$ direction.
 Reeb chords of $\Lambda$ are  in $1-1$ correspondence to double points of $\Pi_{xy}(\Lambda)$, which
 by Ng's algorithm correspond to the crossings and right cusps of $\Pi_{xz}(\Lambda)$.

To define the grading of Reeb chords, we work on the front projection $\Pi_{xz}(\Lambda)$. 
Write $C(\Pi_{xz}(\Lambda))$ for the set of cusps of $\Pi_{xz}(\Lambda)$, 
which divides $\Pi_{xz}(\Lambda)$ into strands (ignoring  double points).
The {\bf Maslov potential} is a function
$\mu$ 
that assigns an integer to each strand 
such that around each cusp, the Maslov potential of the lower strand is one less than that of the upper strand.
This is well defined up to a global shift on each component of $\Lambda$. 
When $n$ is even, we can choose a Maslov potential of the Legendrian  $(2,n)$ torus link such that for any Reeb chord $b_i$ as labeled in Figure \ref{2n}, the upper strand and the lower strand of $b_i$ have the same Maslov potential.
Once the Maslov potential is fixed, the grading of a Reeb chord $c$  that corresponds to a crossing of $\Pi_{xz}(\Lambda)$ can be defined by
$$|c| := \mu(u) - \mu(l),$$
where $u$ is the upper strand of the crossing and $l$ is the lower strand of the crossing.
The grading of  Reeb chords that correspond to  right cusps of $\Pi_{xz}(\Lambda)$ are defined to be $1$.
Extend the definition of grading to $\calA(\Lambda; \bbZ_2[H_1(\Lambda)] )$ by setting $|s_i|=0$ for $i=1, \dots, k$ and using the relation $|ab|=|a|+|b|$.

{\bf Differential.}
The differential $\partial$ is defined by counting rigid holomorphic disks in $\bbR^2_{xy}$ with boundary on $\Pi_{xy}(\Lambda)$.

For any Reeb chords $a, b_1, \dots, b_m$ of $\Lambda$,
define $\calM^{\Lambda}(a; b_1, \dots, b_m)$ to be the moduli space of holomorphic disks:
$$u: (D_{m+1}, \partial D_{m+1}) \to \big(\bbR^2, \Pi_{xy}(\Lambda)\big)$$
with the following properties:
\begin{itemize}
\item $D_{m+1}$ is a $2$-dimensional unit disk with $m+1$ points $s, t_1, \dots, t_m$ removed from the boundary and the points $s, t_1, \dots, t_m$ are labeled in counterclockwise order.
\item  $\displaystyle{\lim_{r \to s}u(r)=a}$ and the neighborhood of $a$ in the image of $u$ covers exactly one positive quadrant of the crossing (see Figure \ref{crossing}).
\item $\displaystyle{\lim_{r \to t_i} u(r)} =b_i$, for $i=1, \dots , m$,
and the neighborhood of $b_i$ in the image of $u$ covers exactly one negative quadrant of the crossing (see Figure \ref{crossing}).
\end{itemize}

\begin{figure}[!ht]
\labellist
\pinlabel  $+$ at 18 23
\pinlabel $+$ at 38 23
\pinlabel $-$ at 28 13
\pinlabel $-$ at 28 33
\endlabellist
\includegraphics[width=1in]{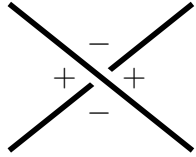}
\caption{At each crossing, the quadrants labeled with $+$ sign are {\bf positive quadrants} and 
the ones labeled with $-$ sign are {\bf negative quadrants}.}
\label{crossing}
\end{figure}

According to \cite{Che}, we have the following dimension formula:
$$\dim \calM^{\Lambda}(a; b_1, \dots, b_m)= |a| - \sum_{i=1}^{m} |b_i|-1.$$
When $\dim \calM^{\Lambda}(a; b_1, \dots, b_m)=0$, the disk $u \in  \calM^{\Lambda}(a; b_1, \dots, b_m)$ is called {\bf rigid}.
There are finitely many rigid holomorphic disks and hence we can count the number of rigid holomorphic disks.

In order to count with  $\bbZ_2[H_1(\Lambda)]$ coefficients, we want to take the homology class of the boundary of  rigid disks in $H_1(\Lambda)$.
However, for any rigid holomorphic disk $u$, the boundary $\Pi_{xy}^{-1}\big(u(\partial D_{m+1})\big)$ is not closed. 
Therefore, we introduce capping paths first.
Equip each connected component $\Lambda_i$ with a reference point $p_i$, for $i=1, \dots, k$. 
For each $i \neq 1$, pick a path $\delta_{1j}$ in $\bbR^3\setminus \Lambda$ that goes from $p_1$ to $p_j$. 
For each Reeb chord $c$ of $\Lambda$ from $c^{-} \in \Lambda_{i^{-}}$ to $c^{+} \in \Lambda_{i^{+}}$, 
the {\bf capping path} $\gamma_{c}$ is defined by concatenating the following four paths:
\begin{itemize}
\item a path on $\Lambda_{i^-}$ from $c^-$ to $p_{i^-}$ ,
\item the chosen path  $-\delta_{1i^{-}}$ connecting $p_{i^-}$ to $p_1$,
\item the chosen path $\delta_{1i^{+}}$ connecting  $p_1$ to $p_{i^+}$, 
\item a path on $\Lambda_{i^+}$ from $p_{i^+}$ to $c^+$.
\end{itemize}

See Figure \ref{capping} for an example of a capping path.

\begin{figure}[Iht]
\labellist
\pinlabel $c$ at 15 52
\pinlabel $c^-$ at -5 45
\pinlabel $c^+$ at -5 55
\pinlabel $p_{2}$ at  105 70
\pinlabel $p_{1}$ at  105 32
\pinlabel $\Lambda_1$ at 122 12
\pinlabel $\Lambda_2$ at 122 90
\endlabellist
\includegraphics[width=2.5in]{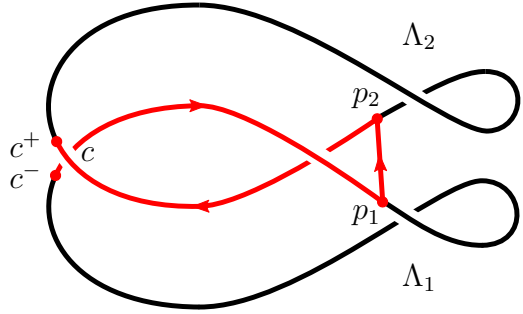}
\caption{Consider the Legendrian Hopf link $\Lambda_1\cup \Lambda_2$. For a Reeb chord $c$ from $c^{-} \in \Lambda_{1}$ to $c^{+} \in \Lambda_{2}$, the red curve is a capping path $\gamma_c$.}
\label{capping}
\end{figure}

After associating each Reeb chord with a capping path, for any rigid holomorphic disk 
$u \in \calM^{\Lambda}(a; b_1, \dots, b_m)$, the curve
$$\tilde{u}= \Pi_{xy}^{-1}\big(u(\partial D_{m+1})\big) \cup \gamma_{a} \cup -\gamma_{b_1} \cup \cdots \cup -\gamma_{b_m}$$
is a loop in $\Lambda \cup \delta_{12} \cup \cdots \cup \delta_{1k}$. 
Notice that $H_1(\Lambda \cup \delta_{12} \cup \cdots \cup \delta_{1k}) \cong H_1(\Lambda)$.
Thus we can view the homology class  $[\tilde{u}]$ as in $ H_1(\Lambda)$. 

Now we are ready to define the differential of the Chekanov-Eliashberg DGA of $\Lambda$.
\begin{defn}
For any Reeb chord $a$ of $\Lambda$, the differential $\partial$ is defined by:
\begin{equation}\label{differential}
\partial(a) = \displaystyle{\sum_{\dim \calM^{\Lambda}(a; b_1, \dots, b_m)=0}\  \sum_{u \in \calM^{\Lambda}(a; b_1, \dots, b_m)}[\tilde{u}]\  b_1\cdots b_m}.
\end{equation}
The definition of differential can be extended to $\calA(\Lambda; \bbZ_2[H_1(\Lambda)])$
by setting $\partial(s_i)=0$ for $i=1, \dots, k$, and using Leibniz rule
$$\partial(ab) =\partial(a)b + a \partial(b).$$
\end{defn}

According to \cite{Che}, the map  $\partial$ is a differential in degree $-1$. 
Moreover,
up to stable tame isomorphism, the Chekanov-Eliashberg DGA $\big(\calA(\Lambda; \bbZ_2[H_1(\Lambda)]), \partial\big)$ is an invariant of $\Lambda$ under Legendrian isotopy.

\begin{rmk}
In general, for any commutative ring $R$ and a ring homomorphism $\bbZ_2[H_1(\Lambda)]\to R$,  we define the Chekanov-Eliashberg DGA $\big(\calA(\Lambda; R), \partial \big)$ as 
 a tensor product of the DGA $\calA(\Lambda;\bbZ_2[H_1(\Lambda)])$ with the ring $R$:
$$\calA(\Lambda; R)= \calA(\Lambda;\bbZ_2[H_1(\Lambda)]) \otimes_{\bbZ_2[H_1(\Lambda)]} R,$$
where the ring homomorphism gives $R$ the structure of a module over $\bbZ_2[H_1(\Lambda)]$.
\end{rmk}

Now we give a combinatorial definition of the differential of $\big(\calA(\Lambda; \bbZ_2[H_1(\Lambda)]), \partial \big)$.
Assign $\Lambda$ an orientation and label  each component $\Lambda_i$, for $i=1, \dots, k$,  with a base point $s_i$, which is different from the reference point  and ends of Reeb chords. 
For a union of oriented curves $\gamma$ in $\Lambda \cup \delta_{12} \cup \cdots \cup \delta_{1k}$,
we associate it with a monomial $w(\gamma)$ in $\bbZ_2[H_1(\Lambda)]$ 
\begin{equation}\label{mon}
w(\gamma) = \displaystyle{\prod_{i=1}^{k}s_i^{n_i(\gamma)}},
\end{equation}
 where $n_i(\gamma)$ is the number of times $\gamma$ goes through $s_i$ counted with sign.
The sign is positive if  $\gamma$ goes through $s_i$ following  the link orientation and is negative if $\gamma$ goes through $s_i$ against the link orientation. 
In particular, for a rigid holomorphic disk $u \in \calM^{\Lambda}(a; b_1, \dots, b_m)$, we have 
$$[\tilde{u}]= w(\tilde{u})=w(u)w(\gamma_{a})\prod_{i=1}^{m} w(\gamma_{b_i})^{-1},$$
where $w(u)$ is short for $w\big(\Pi_{xy}^{-1}(u(\partial D_{m+1}))\big).$
Plugging it into the formula \eqref{differential}, we get a combinatorial definition of the differential.
It seems to depend on the choice of capping paths. 
However, we have the following well-known proposition.
\begin{prop} \label{invariant}
Let $\Lambda$ be a Legendrian link and $\gamma$, $\gamma'$ be two families of capping paths of Reeb chords of $\Lambda$.
The corresponding DGAs $\big(\calA^{\gamma}(\Lambda), \partial\big)$ and $\big(\calA^{\gamma'}(\Lambda), \partial'\big)$ are isomorphic.
\end{prop} 
\begin{proof}
For a Reeb chord $a$ of $\Lambda$, we have
$$\partial(a) = \displaystyle{\sum_{\dim \calM^{\Lambda}(a; b_1, \dots, b_m)=0}  \sum_{u \in \calM^{\Lambda}(a; b_1, \dots, b_m)} \left(w(u)w(\gamma_{a})\prod_{i=1}^{m} w(\gamma_{b_i})^{-1}\right)\ b_1\cdots b_m},$$
$$\partial'(a) = \displaystyle{\sum_{\dim \calM^{\Lambda}(a; b_1, \dots, b_m)=0} \sum_{u \in \calM^{\Lambda}(a; b_1, \dots, b_m)} \left( w(u)w(\gamma'_{a})\prod_{i=1}^{m} w(\gamma'_{b_i})^{-1}\right) \ b_1\cdots b_m}.$$

For each Reeb chord $c$,  concatenate $-\gamma'_c$ with $\gamma_c$  and get a closed curve, denoted
by $-\gamma'_{c}\cup \gamma_c$. 
It is not hard to check that the map
$$
\begin{array}{rll}
f: (\calA^{\gamma}(\Lambda), \partial) & \to & (\calA^{\gamma'}(\Lambda), \partial')\\
&&\\
c & \mapsto & [-\gamma'_c\cup \gamma_c]\  c=w(\gamma'_{c})^{-1}w(\gamma_{c})\ c\\
\end{array}
$$ 
 is  a chain map and is an isomorphism.
\end{proof}

Note that for an oriented link $\Lambda$ with minimal base points (i.e. each component has exactly one base point), 
we can choose a family of capping paths such that none of them pass through any base point. 
Therefore, we only need to count  intersections of the disk boundary and  base points.
Thanks to Proposition \ref{invariant}, we can define the Chekanov-Eliashberg DGA of $\Lambda$ to be
 a unital graded algebra over 
$\bbZ_2[H_1(\Lambda)]=\bbZ_2[s_1^{\pm  1}, \dots, s_k^{\pm  1}]$ generated by Reeb chords of $\Lambda$ endowed with a differential given by
$$
\begin{array}{rl}
\partial(a) &= \displaystyle{\sum_{\dim \calM^{\Lambda}(a; b_1, \dots, b_m)|=0}  \sum_{u \in \calM^{\Lambda}(a; b_1, \dots, b_m)}w(u)b_1\cdots b_m},\\
&\\

\partial(s_i)&=0, \hspace{0.1in} i= 1, \dots, k,\\
\end{array}
$$
where $w(u)$ is defined by formula \eqref{mon}. 
This DGA is denoted by $\big(\calA(\Lambda, \{s_1, \dots, s_k\}), \partial\big)$ as well.

\begin{eg}
\begin{figure}[!ht]
\labellist
\pinlabel $b_1$ at 12 30
\pinlabel $b_2$ at 55 30
\pinlabel $b_3$ at 105 30
\pinlabel $a_1$ at 125 50
\pinlabel $a_2$ at 125 18
\pinlabel $s$ at  135 73
\endlabellist
\includegraphics[width=2in]{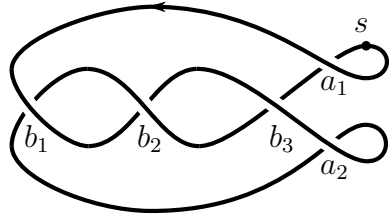}
\caption{The Lagrangian projection of the Legendrian $(2,3)$ torus knot with a single  base point.}
\label{23}
\end{figure}
For the Legendrian $(2,3)$ torus knot $\Lambda$ with a single base point $s$ as shown in Figure \ref{23}. 
The underlying algebra $\calA(\Lambda, \{s\})$ is generated by Reeb chords $a_1, a_2, b_1, b_2, b_3$  over $\bbZ_2[s, s^{-1}]$.
Reeb chords $a_1$ and $a_2$ are in degree $1$ and the rest of Reeb chords are in degree $0$.
The differential is given by:
$$
\def\arraystretch{1.5}
\begin{array}{rll}
\partial(a_1)&=& s^{-1} + b_1 + b_3 + b_1b_2b_3,\\
\partial(a_2)&=& 1+b_1+b_3 + b_3b_2b_1,\\
\partial(b_i)&=&0, \hspace{0.2in} i=1,2,3,\\
\partial(s)&=&\partial(s^{-1})=0.
\end{array}
$$

\end{eg}
The definition of DGA of Legendrian link can be generalized to the case where there are more than one base point on  some components of the link.
Let $\Lambda$ be an oriented Legendrian link and $\{s_1, \dots, s_l\}$ be a set of points on $\Lambda$ such that each component of $\Lambda$ has at least one point in the set
and the set does not include any end of any Reeb chord of $\Lambda$.
For a union of paths $\gamma$, associate it with a monomial
 $w(\gamma) = \displaystyle{\prod_{j=1}^{l}s_j^{n_j(\gamma)}}$  in $\bbZ_2[s_1^{\pm1}, \dots, s_l^{\pm1}]$, where $n_j$ is defined similar as above. 
The DGA $\big(\calA(\Lambda, \{s_1, \dots, s_l\}), \partial\big)$ is a unital graded algebra 
generated by Reeb chords of $\Lambda$ over $\bbZ_2[s_1^{\pm 1}, \dots, s_l^{\pm 1}]$
endowed with a differential
given by
$$
\begin{array}{rl}
\partial(a) &= \displaystyle{\sum_{\dim \calM^{\Lambda}(a; b_1, \dots, b_m)|=0}  \sum_{u \in \calM^{\Lambda}(a; b_1, \dots, b_m)}w(u)b_1\cdots b_m},\\
&\\

\partial(s_i)&=0, \hspace{0.1in} i=1,\dots, l.\\
\end{array}
$$

\subsection{The DGA map induced by exact Lagrangian cobordisms}\label{DGAmap}
According to \cite{EHK}, the Chekanov-Eliashberg DGA acts functorially on  exact Lagrangian cobordisms. 
We first recall the definition of exact Lagrangian cobordisms.

\begin{defn}
Let $\Lambda_+$ and $\Lambda_-$ be Legendrian submanifolds in $(\bbR^3, \ker \alpha)$,  where $\alpha=dz-y dx$.
An {\bf exact Lagrangian cobordism} $\Sigma$ from $\Lambda_-$ to $\Lambda_+$ is 
a $2$-dimensional surface in $\big(\bbR\times \bbR^3, d(e^t\alpha)\big)$ such that 
there exists $T>0$ such that  $\Sigma$ is
\begin{itemize}
\item cylindrical over $\Lambda_+$ on the positive end, i.e. $\Sigma \cap \big((T,\infty)\times \bbR^3\big)= (T,\infty) \times \Lambda_+$;
\item cylindrical over $\Lambda_-$ on the negative end, i.e. $\Sigma \cap \big((-\infty,-T)\times \bbR^3\big)= (-\infty,-T) \times \Lambda_-$;
\item compact in $[-T, T]\times \bbR^3$,
\end{itemize}
and $e^t\alpha \big\vert_{T\Sigma}=df$\vspace{0.02in} for some function $f: \Sigma \to \bbR$. (See Figure \ref{lagcob}.)

When $\Lambda_-$ is empty, the surface $L$ satisfying the conditions above  is called an {\bf
exact Lagrangian filling} of $\Lambda_+$.
\end{defn}

\begin{figure}[!ht]
\begin{minipage}{3in}
\begin{center}
\labellist
\small

\pinlabel $t$ at -14 350
\pinlabel $\Lambda_+$  at 345 310
\pinlabel $\Sigma$  at 345 220
\pinlabel $\Lambda_-$ at 345 120

\pinlabel $N$ at -14 250
\pinlabel $-N$ at -24 50
\endlabellist

\includegraphics[width=2in]{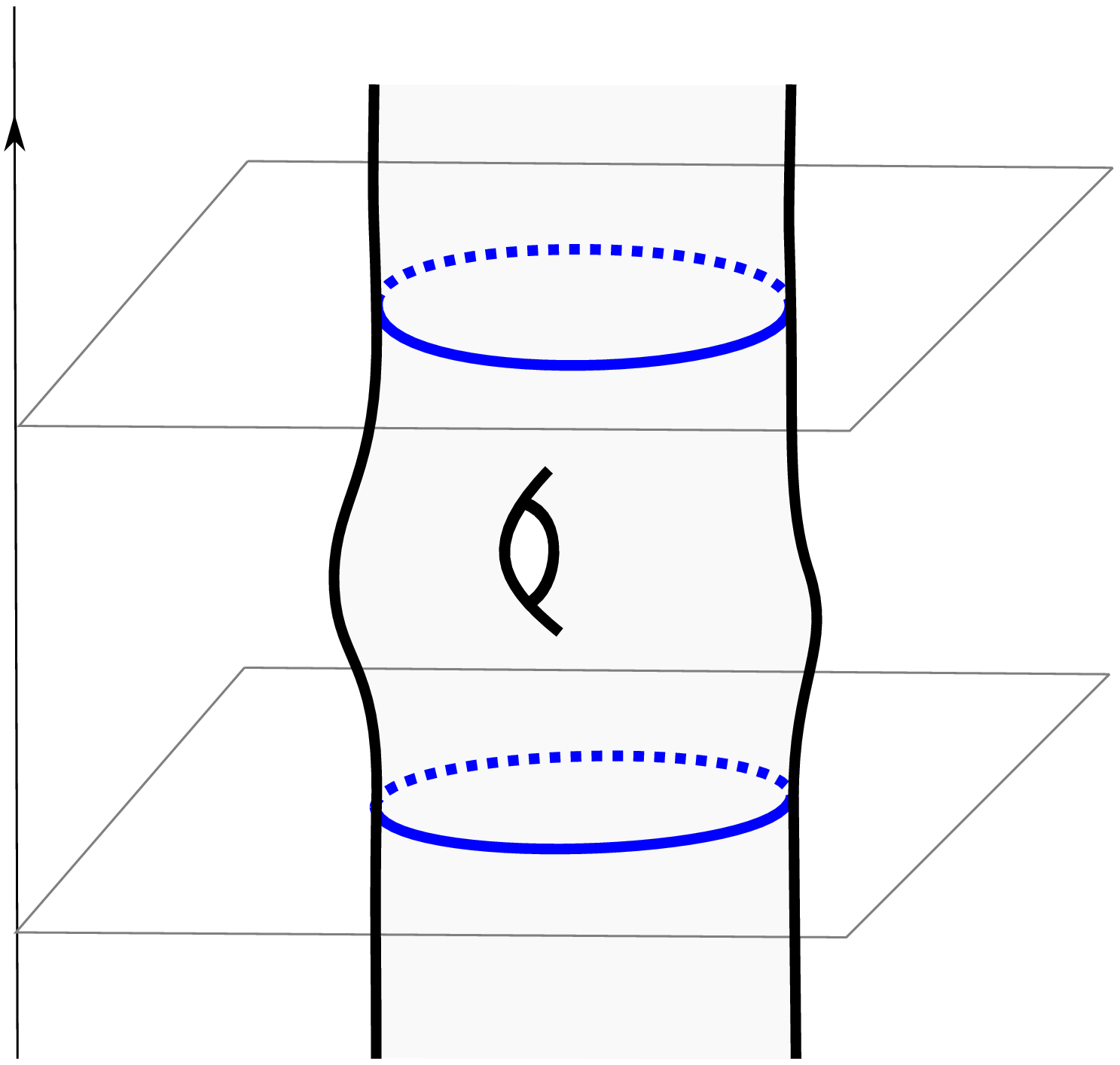}
\end{center}
\labellist
\small

\endlabellist

\caption{A schematic picture of an exact Lagrangian cobordism.}
\label{lagcob}
\end{minipage}
\begin{minipage}{3in}
\begin{center}
\labellist
\small
\pinlabel $\overline{\Sigma}_-$ at 12 94
\pinlabel $\overline{\Sigma}_+$ at 38 123
\pinlabel $\Sigma$ at 35 52
\pinlabel $\Lambda_+$ at  135 93
\pinlabel $\Lambda_-$ at 132 10
\endlabellist
\includegraphics[width=1.5in]{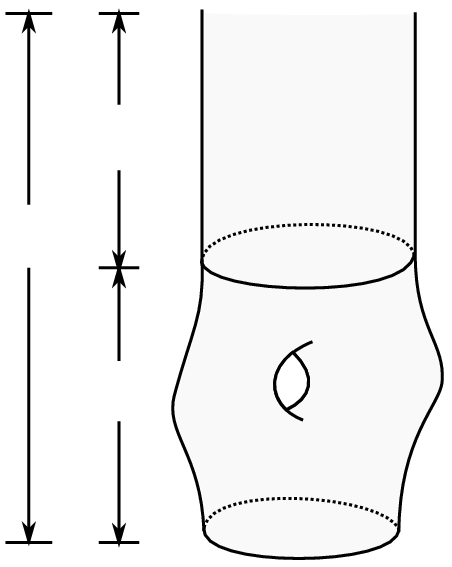}
\end{center}
\caption{The relation among cobordisms $\overline{\Sigma}_+$, $\overline{\Sigma}_-$, and $\Sigma$. }
\label{concatenation}
\end{minipage}
\end{figure}

By \cite{EHK}, an exact Lagrangian cobordism $\Sigma$ from $\Lambda_-$ to $\Lambda_+$ gives a DGA map from $\calA(\Lambda_+)$ to $\calA(\Lambda_-)$ with $\bbZ_2[H_1(\Sigma)]$ coefficients.
Therefore, an exact Lagrangian filling $L$ of a Legendrian submanifold $\Lambda$, which can be viewed as a cobordism from the empty set to $\Lambda$, gives a DGA map from $\calA(\Lambda)$ to 
the trivial DGA $(\bbZ_2[H_1(L)],0 )$, which is an augmentation of  $\calA(\Lambda)$ 
to 
$\bbZ_2[H_1(L)]$.

For the purpose of  computing augmentations of the Legendrian $(2,n)$ torus knots in Section \ref{aug}, we  revise the coefficient ring of the DGA map induced by  exact Lagrangian cobordisms from \cite{EHK}.
In stead of using $\bbZ_2[H_1(\Sigma)]$ coefficients, we will show the following proposition:

\begin{prop} \label{map}
Let  $\Lambda_+$ and $\Lambda_-$ be Legendrian submanifolds in $(\bbR^3, \ker \alpha)$
and $\Sigma$ be a connected exact Lagrangian cobordism from $\Lambda_-$ to $\Lambda_+$.
Assume that
$\overline{\Sigma}_+$ is a connected exact Lagrangian cobordism from  $\Lambda_+$ to some other Legendrian link and
 $\overline{\Sigma}_-$ is the concatenation of $ \overline{\Sigma}_+$ and $\Sigma$
 as shown in Figure \ref{concatenation}. 
 The exact Lagrangian cobordism $\Sigma$ induces a  DGA map
 $$\Phi: \big(\calA(\Lambda_+; \bbZ_2[H_1(\overline{\Sigma}_+)]), \partial_+\big) \to \big(\calA(\Lambda_-; \bbZ_2[H_1(\overline{\Sigma}_-)]), \partial_-\big) $$
 with $\bbZ_2[H_1(\overline{\Sigma}_-)]$ coefficients.
 \end{prop}
 
 Note that when $\overline{\Sigma}_+$ is an exact Lagrangian cylinder over $\Lambda_+$, this map agrees with the DGA map introduced by \cite{EHK}.
 The proof of Proposition \ref{map}  follows \cite[Section 3]{EHK}.
 Our revision of the coefficient ring is based on a different choice of capping paths of $\Lambda_+$ and $\Lambda_-$.
 The capping paths of are chosen on $\Sigma$ in \cite{EHK} while we choose capping paths of $\Lambda_+$ on $\overline{\Sigma}_+$ and capping paths of $\Lambda_-$ on $\overline{\Sigma}_-$.
 For the rest of the section, we will describe this DGA map.

The inclusion map $\Lambda_+ \hookrightarrow \overline{\Sigma}_+$ makes it natural to define the DGA $\big(\calA(\Lambda_{+}; \bbZ_2[H_1(\overline{\Sigma}_{+})]), \partial_+\big)$.
The underlying algebra  $$\calA(\Lambda_{+}; \bbZ_2[H_1(\overline{\Sigma}_{+})])= \calA(\Lambda_{+}; \bbZ_2[H_1(\Lambda_+)])\otimes_{\bbZ_2[H_1(\Lambda_+)]}  \bbZ_2[H_1(\overline{\Sigma}_{+})]$$
is generated by Reeb chords of $\Lambda_+$ over the ring $\bbZ_2[H_1(\overline{\Sigma}_{+})]$.
Given that $\overline{\Sigma}_{+}$ is connected, we can choose a family of capping paths for $\Lambda_{+}$ on $\overline{\Sigma}_{+}$. Therefore, for any rigid holomorphic disk $u_{+}$ counted by $\partial_{+}$, it is natural to take the homology class of $\tilde{u}_{+}$ in $H_1(\overline{\Sigma}_{+})$.
Hence the differential coefficients of $\partial_{+}$ are in $\bbZ_2[H_1(\overline{\Sigma}_{+})]$.
In addition, the DGA  $\big(\calA(\Lambda_{+}; \bbZ_2[H_1(\overline{\Sigma}_{+})]), \partial_+\big)$ does not depend on the choice of capping paths on $\overline{\Sigma}_{+}$ for a similar reason as Proposition \ref{invariant}.
The DGA $\big(\calA(\Lambda_{-}; \bbZ_2[H_1(\overline{\Sigma}_{-})]), \partial_-\big)$ is defined similarly.

The DGA  map $\Phi$ induced by $\Sigma$ is a composition of two maps.
The first map 
$$\psi: \big(\calA(\Lambda_+; \bbZ_2[H_1(\overline{\Sigma}_+)]), \partial_+\big) \to \big(\calA(\Lambda_+; \bbZ_2[H_1(\overline{\Sigma}_-)]), \partial_+\big)$$
is induced by the inclusion map $\overline{\Sigma}_+ \hookrightarrow \overline{\Sigma}_-$.
It is not hard to show $\psi$ is a DGA map.
The second map
$$\phi : \big(\calA(\Lambda_+; \bbZ_2[H_1(\overline{\Sigma}_-)]), \partial_+\big) \to \big(\calA(\Lambda_-; \bbZ_2[H_1(\overline{\Sigma}_-)]), \partial_-\big)$$
is defined by counting rigid holomorphic disks in $\bbR\times \bbR^3$ with boundary on $\Sigma$.

Fix an almost complex structure $J$ on $\bbR\times\bbR^3$ which is adjusted to the symplectic form $\omega$ (see  \cite[Section 3.2]{EHK} for details).
For a Reeb chord $a$ of $\Lambda_+$ and Reeb chords $b_1, \dots , b_m$ of $\Lambda_-$, 
define $\calM^{\Sigma}(a; b_1,\dots, b_m)$ to be the moduli space of $J$-holomorphic disks:
$$u: (D_{m+1}, \partial D_{m+1} ) \to (\bbR\times \bbR^3, \Sigma)$$
with the following properties:
\begin{itemize}
\item $D_{m+1}$ is a $2$-dimensional unit disk with $m+1$ points $r, s_1, s_2, \dots , s_m$ removed. 
The points $r, s_1, s_2, \dots , s_m$ are arranged on the boundary of the disk counterclockwise.
\item The image of $u$  is asymptotic to a strip $\bbR_+ \times a$ around $r$.
\item The image of $u$  is asymptotic to a strip  $\bbR_- \times b_i$ around $s_i$ for $i=1, \dots , m$.
\end{itemize}

By \cite{CEL}, there is a corresponding dimension formula:
$$\dim \calM^{\Sigma}(a; b_1, \dots, b_m)= |a|-\sum_{i=1}^m|b_i|.
$$
When $\dim \calM^{\Sigma}(a; b_1, \dots, b_m)=0$, the $J$-holomorphic disk $u \in \calM^{\Sigma}(a; b_1,\dots, b_m)$
is called {\bf rigid}.
For each rigid $J$-holomorphic disk $u$, 
  concatenate the image of the disk boundary with the capping paths of corresponding Reeb chords on $\overline{\Sigma}_-$ and get 
$$\tilde{u} = u(\partial D_{m+1}) \cup \gamma_a \cup -\gamma_{b_1} \cup \cdots \cup -\gamma_{b_m},$$ 
which is a loop in $\overline{\Sigma}_-$. 
Hence we can take the homology class of $\tilde{u}$ in $H_1(\overline{\Sigma}_-)$, denoted by $[\tilde{u}]_{\overline{\Sigma}_-}$.
The  map 
$$\phi : \big(\calA(\Lambda_+; \bbZ_2[H_1(\overline{\Sigma}_-)]), \partial_+\big) \to \big(\calA(\Lambda_-; \bbZ_2[H_1(\overline{\Sigma}_-)]), \partial_-\big)$$
is defined as follows.
For any  Reeb chord $a$ of $\Lambda_+$, the map $\phi$ maps $a$ to 
$$\phi(a) = \displaystyle{\sum_{\dim \calM^{\Sigma}(a; b_1, \dots, b_m)=0} \sum_{ u \in \calM^{\Sigma}(a; b_1, \dots, b_m)}[u]_{\overline{\Sigma}_-} b_1\cdots b_m}.$$
The map $\phi$ is identity on $\bbZ_1[H_1(\overline{\Sigma}_-)]$.
By \cite[Section 3.5]{EHK}, the map $\phi$ is a DGA map.

Therefore, the exact Lagrangian cobordism $\Sigma$ induces a DGA map $\Phi=\phi \circ  \psi$  
$$\Phi: \big(\calA(\Lambda_+; \bbZ_2[H_1(\overline{\Sigma}_+)]), \partial_+\big) \to \big(\calA(\Lambda_-; \bbZ_2[H_1(\overline{\Sigma}_-)]), \partial_-\big).$$

\vspace{0.2in}

\section{Main Results}\label{results}
We consider the exact Lagrangian fillings of the Legendrian $(2,n)$ torus knot contructed from the \cite{EHK} algorithm.
Each filling can be achieved by concatenating $n$ successive saddle cobordisms with two minimum cobordisms.
In Section \ref{aug}, we combine results in \cite{EHK} and  Proposition \ref{map} to write down  combinatorial formulas for the DGA maps induced by a pinch move and a minimum cobordism.
Composing all the DGA maps induced by $n$ ordered pinch moves and the two minimum cobordisms,
we obtain a combinatorial formula for  augmentations of $\calA(\Lambda)$ to $\bbZ_2[H_1(L)]$ induced by exact Lagrangian fillings $L$.
In Section \ref{equi},
we find a combinatorial invariant to distinguish these resulting augmentations and hence we show that the $C_n$ exact Lagrangian fillings are distinct up to exact Lagrangian isotopy.
As a corollary, we extend the result to the case $n$ is even.

\subsection {Computation of augmentations} \label{aug}

Consider the Lagrangian projection of the Legendrian $(2,n)$ torus knot $\Lambda$ with a base point $\tilde{s}_0$ and label the $n$ crossings in degree $0$ from left to right by $b_1, \dots , b_n$ as shown in Figure \ref{2nast}. 
\begin{figure}[!ht]
\labellist
\pinlabel  $b_1$ at 16 47
\pinlabel  $b_2$ at 47 47
\pinlabel $b_n$ at  252 47
\pinlabel $a_1$ at 273 115
\pinlabel $a_2$ at 275 23
\pinlabel $\tilde{s}_0$ at 310 113
\endlabellist
\includegraphics[width=2.8in]{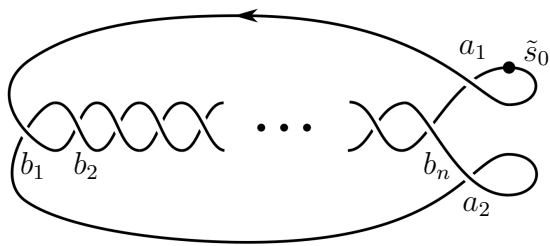}
\caption{The Lagrangian projection of the Legendrian $(2,n)$ torus knot with a base point.}
\label{2nast}
\end{figure}

For each permutation $\sigma$ of $\{1, \dots, n\}$,
the corresponding exact Lagrangian filling $L_{\sigma}$ of the Legendrian $(2,n)$ torus knot $\Lambda$ 
is achieved in the following way:
\begin{itemize}
\item
Start with an exact Lagrangian cylinder over $\Lambda$, denoted by $\overline{\Sigma}_0$. Label $\Lambda$ as $\Lambda_0$.
\item
For $i=1, \dots, n$, concatenate $\overline{\Sigma}_{i-1}$ from the bottom with a saddle cobordism $\Sigma_i$ corresponding to the pinch move at crossing $b_{\sigma(i)}$ and get a new exact Lagrangian cobordism $\overline{\Sigma}_{i}$. Label the new Legendrian submanifold after pinch move as $\Lambda_i$.
\item
Finally, use two minimal cobordisms, denoted by $\Sigma_{n+1}$,  to close up $\overline{\Sigma}_n$ from the bottom and get the exact Lagrangian filling $L_{\sigma}$. 
To be consistent, let $\Lambda_{n+1}$ be the empty set.
\end{itemize}

By Proposition \ref{map}, for $i=1, \dots, n+1$,
 each exact Lagrangian cobordism $\Sigma_i$ induces a DGA map: 
 $$\Phi_i: \big(\calA(\Lambda_{i-1}; \bbZ_2[H_1(\overline{\Sigma}_{i-1})]), \partial_{i-1}\big) \to \big(\calA(\Lambda_i; \bbZ_2[H_1(\overline{\Sigma}_i)]), \partial_i\big).$$
The map $\Phi_{n+1}$ that is induced by  minimum cobordisms is well understood
while the maps $\Phi_i$ for $i=1, \dots, n$ that correspond to pinch moves are not.
We will first study $H_1(\overline{\Sigma}_n)$ and give a geometric description of the DGA map that corresponds to a pinch move.
Combining with \cite{EHK}, we will write down an explicit combinatorial formula for each $\Phi_i$, for $i=1, \dots, n+1$.
 
To describe $H_1(\overline{\Sigma}_n)$ easily, we chop off the cylindrical top of $\overline{\Sigma}_n$ and view it as a surface with boundary $\Lambda\cup \Lambda_n$, denoted by $\overline{\Sigma}_n$ as well. 
By Poincar{\'e} duality, we have $H^1(\overline{\Sigma}_n) \cong H_1(\overline{\Sigma}_n, \Lambda \cup \Lambda_n)$. 
In particular,
for each oriented curve $\alpha$ in $\overline{\Sigma}_n$ with ends on $\Lambda \cup \Lambda_n$, which is an element in $H_1(\overline{\Sigma}_n, \Lambda \cup \Lambda_n)$, there exists an element $\theta_{\alpha} \in H^1(\overline{\Sigma}_n)$
such that for any oriented loop $\beta$ in $\overline{\Sigma}_n$, the intersection number of $\alpha$ and $\beta$ is $\theta_{\alpha}(\beta)$.
Thus, in order to know the homology class of a loop $\beta$ in $H_1(\overline{\Sigma}_n)$, 
we only need to count the intersection number of each generator curve of $H_1(\overline{\Sigma}_n, \Lambda \cup \Lambda_n)$ with $\beta$.

We choose the set of generator curves of $H_1(\overline{\Sigma}_n, \Lambda \cup \Lambda_n)$ as follows.
Use $t$ coordinate to slice $\overline{\Sigma}_n$ into a movie of diagrams (some of them may not be Legendrian diagrams).
We study the trace of points on the diagram when $t$ is decreasing.
For $i=1, \dots, n$, the saddle cobordism $\Sigma_i$ flows all the points directly downward except ends of the Reeb chord $b_{\sigma(i)}$.
According to \cite{Lin}, the ends of the Reeb chord $b_{\sigma(i)}$ merge to a point $r_{\sigma(i)}$, and then split into two points, labeled as $\tilde{s}_{\sigma(i)}$ and $\tilde{s}_{\sigma(i)}^{-1}$ respectively.
Now for $i=1,\dots, n$, consider the trace of $\tilde{s}$ in $\overline{\Sigma}_n$, which is a flow line from $r_i$ to the bottom of $\overline{\Sigma}_n$.
Concatenating it with the inverse trace of $\tilde{s}_{i}^{-1}$ in $\overline{\Sigma}_n$, we get a curve $\alpha_i$ in $\overline{\Sigma}_n$  as shown in Figure \ref{trace}.
In addition, denote the trace of the base point $\tilde{s}_0$ in $\overline{\Sigma}_n$ by $\alpha_{0}$.
In this way, we have that $\alpha =\{\alpha_0, \alpha_{1}, \dots, \alpha_{n}\}$ is a set of generator curves of $H_1(\overline{\Sigma}_n, \Lambda \cup \Lambda_n)\cong \bbZ^{n+1}$.

\begin{figure}[!ht]
\labellist
\pinlabel $r_2$ at 55 120
\pinlabel $\alpha_2$ at 35 90
\pinlabel $\alpha_0$ at 110 130
\pinlabel $\tilde{s}_0$ at 115 200 
\pinlabel $\tilde{s}_0$ at 112 62
\pinlabel $\tilde{s}^{-1}_2$ at 53 18
\pinlabel $\tilde{s}_2$ at  64 50
\pinlabel $b_2$ at 53 185
\pinlabel $\overline{\Sigma}_1$ at -10  100 
\endlabellist
\includegraphics[width=1.8in]{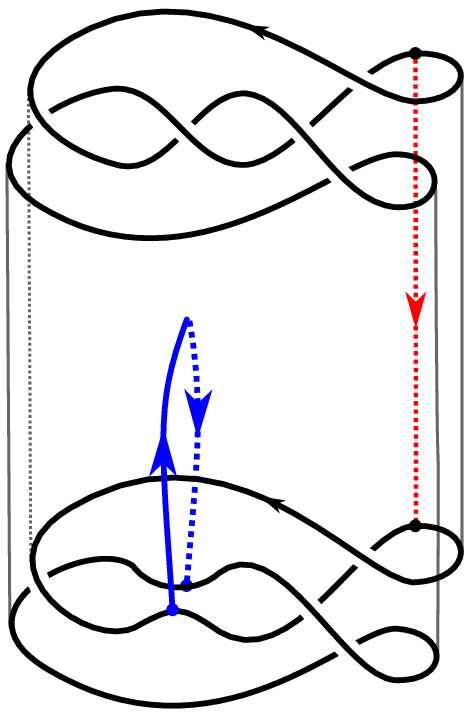}
\caption{As an example, assume $\Lambda$ is the Legendrian $(2,3)$ torus knot and the first pinch move  is taken at $b_2$.  The blue curve and the red curve are  $\alpha_2$ and $\alpha_0$ restricted on $\overline{\Sigma}_1$, respectively.}
\label{trace}
\end{figure}

For each curve $\alpha_i$, where $i=0, \dots, n$, Poincar{\'e} duality gives an element $\theta_{\alpha_i}\in H^1(\overline{\Sigma}_n)$. 
Denote its dual in $H_1(\Sigma_n)$ by $\tilde{s}_i$.
Therefore, for any union of paths $\gamma$ in $\overline{\Sigma}_n$, the monomial $w(\gamma)$ associated to $\gamma$ in $\bbZ_2[H_1(\overline{\Sigma}_n)]$ 
is 
$$w(\gamma) =\prod\limits_{i=0}^{n}\tilde{s}_i^{n_i(\gamma)},$$ 
where $n_i(\gamma)$ is intersection number of  $\alpha_i$ and  $\gamma$ counted with signs.

For $i<n$, the map $H_1(\overline{\Sigma}_i) \to H_1(\overline{\Sigma}_{n})$ induced by the inclusion map is injective.
A similar argument shows that for a union of paths $\gamma$ in $\overline{\Sigma}_i$, the monomial associated to $\gamma$ in $\bbZ_2[H_1(\overline{\Sigma}_i)]$ counts intersections of $\alpha_0, \alpha_{\sigma(1)}, \dots, \alpha_{\sigma(i)}$ with $\gamma$.
Notice that the curves $\alpha_{\sigma(i+1)}, \dots, \alpha_{\sigma(n)}$ do not intersect $\overline{\Sigma}_i$.
Hence the monomial  in $\bbZ_2[H_1(\overline{\Sigma}_i)]$ agrees with $w(\gamma)$ in $\bbZ_2[H_1(\overline{\Sigma}_n)]$.

Choose a family of capping paths for $\Lambda_i$ on $\overline{\Sigma}_i$ for $i=0, \dots, n$.
By Proposition \ref{map}, for $i=1, \dots, n+1$, each exact Lagrangian cobordism $\Sigma_i$
gives a  DGA map $\Phi_i$:
$$
\Phi_i: \big(\calA(\Lambda_{i-1}; \bbZ_2[H_1(\overline{\Sigma}_{i-1})]), \partial_{i-1}\big) \to
 \big(\calA(\Lambda_i; \bbZ_2[H_1(\overline{\Sigma}_i)]), \partial_i\big),
 $$
which maps any Reeb chord $a$ of $\Lambda_{i-1}$ to
$$ 
\begin{array}{rl}
& \displaystyle{\sum_{\dim \calM^{\Sigma_i}(a; b_1, \dots, b_m)=0}  \sum_{u \in \calM^{\Sigma_i}(a; b_1, \dots, b_m)}w(\tilde{u})\ b_1\cdots b_m}\\
&\\
=& 
\displaystyle{\sum_{\dim \calM^{\Sigma_i}(a; b_1, \dots, b_m)=0}  \sum_{u \in \calM^{\Sigma_i}(a; b_1, \dots, b_m)}\left(w(\gamma_{a})w(u)\prod_{i=1}^{m} w(\gamma_{b_i})^{-1}\right)\  b_1\cdots b_m}.
\end{array}
 $$

Now we show that the DGA map induced by the exact Lagrangian cobordisms  is independent of the choice of capping paths.

\begin{thm}
Let $\gamma$ and $\gamma'$ be two families of capping paths of $\Lambda_i$ on $\overline{\Sigma}_i$ for $i=0, \dots, n$.
Denote the corresponding DGAs by
$\big(\calA^{\gamma}(\Lambda_{i}; \bbZ_2[H_1(\overline{\Sigma}_{i})]), \partial^{\gamma}_i\big)$ and $\big(\calA^{\gamma'}(\Lambda_{i}; \bbZ_2[H_1(\overline{\Sigma}_{i})]), \partial^{\gamma'}_i\big)$.
Assume $\Phi^{\gamma}_i$  and $\Phi^{\gamma'}_i$ are the corresponding the DGA maps induced by $\Sigma_i$.
Then the maps
$$
\begin{array}{rll}
f_i: \big(\calA^{\gamma}(\Lambda_{i}; \bbZ_2[H_1(\overline{\Sigma}_{i})]), \partial^{\gamma}_i\big)& \to& \big(\calA^{\gamma'}(\Lambda_{i}; \bbZ_2[H_1(\overline{\Sigma}_{i})]), \partial^{\gamma'}_i\big)\\
&&\\
c & \mapsto & w(\gamma'_c)^{-1}w(\gamma_c)\ c\\
\end{array}
$$
are DGA isomorphisms  for $i=0, \dots, n$.
Moreover, the following diagram commutes:

\diag{\big(\calA^{\gamma}(\Lambda_{i-1}; \bbZ_2[H_1(\overline{\Sigma}_{i-1})]), \partial^{\gamma}_{i-1}\big) \arr^{f_{i-1}} \ard_{\Phi^{\gamma}_i}& 
\big(\calA^{\gamma'}(\Lambda_{i-1}; \bbZ_2[H_1(\overline{\Sigma}_{i-1})]), \partial^{\gamma'}_{i-1}\big)\ard^{\Phi^{\gamma'}_i}\\
\big(\calA^{\gamma}(\Lambda_{i}; \bbZ_2[H_1(\overline{\Sigma}_{i})]), \partial^{\gamma}_i\big) \arr^{f_{i}} & 
\big(\calA^{\gamma'}(\Lambda_{i}; \bbZ_2[H_1(\overline{\Sigma}_{i})]), \partial^{\gamma'}_i\big).}

\end{thm}

\begin{proof}
The maps $f_i$ are DGA isomorphisms for the same reason as Proposition \ref{invariant}.
Now we prove the second part. For any Reeb chord $a$ of $\Lambda_{i-1}$,
$$
\small\def\arraystretch{2.5}
\begin{array} {rl}
f_i \circ \Phi^{\gamma}_{i} (a) & =f_i \left( \displaystyle{\sum_{\dim \calM^{\Sigma_i}(a; b_1, \dots, b_m)=0}  \sum_{u \in \calM^{\Sigma_i}(a; b_1, \dots, b_m)}\left(w(\gamma_{a})w(u)\prod_{i=1}^{m} w(\gamma_{b_i})^{-1}\right)\  b_1\cdots b_m}\right)\\
&= \displaystyle{\sum_{\dim \calM^{\Sigma_i}(a; b_1, \dots, b_m)=0}  \sum_{u \in \calM^{\Sigma_i}(a; b_1, \dots, b_m)}
\left(w(\gamma_{a})w(u)\prod_{i=1}^{m} w(\gamma_{b_i})^{-1}w(\gamma'_{b_i})^{-1}w(\gamma_{b_i})\right) b_1\cdots b_m}\\
&= \displaystyle{\sum_{\dim \calM^{\Sigma_i}(a; b_1, \dots, b_m)=0}  \sum_{u \in \calM^{\Sigma_i}(a; b_1, \dots, b_m)}\left(w(\gamma_{a})w(u)\prod_{i=1}^{m} w(\gamma'_{b_i})^{-1}\right)\  b_1\cdots b_m},\\
\Phi^{\gamma'}_i \circ f_{i-1}(a) & =\Phi^{\gamma'}_i \Big(w(\gamma'_{a})^{-1}w(\gamma_{a})a\Big)\\
&=w(\gamma'_{a})^{-1}w(\gamma_{a}) \displaystyle{\sum_{\dim \calM^{\Sigma_i}(a; b_1, \dots, b_m)=0}  \sum_{u \in \calM^{\Sigma_i}(a; b_1, \dots, b_m)} \left(w(\gamma'_{a})w(u)\prod_{i=1}^{m} w(\gamma'_{b_i})^{-1}\right) b_1\cdots b_m}\\
&=   \displaystyle{\sum_{\dim \calM^{\Sigma_i}(a; b_1, \dots, b_m)=0}  \sum_{u \in \calM^{\Sigma_i}(a; b_1, \dots, b_m)}\left(w(\gamma_{a})w(u)\prod_{i=1}^{m} w(\gamma'_{b_i})^{-1}\right)\  b_1\cdots b_m}.\\
\end{array}
$$

\end{proof}

Note that, if we cut $\overline{\Sigma}_i$ along the curves $\alpha_0, \alpha_{\sigma(1)}, \dots, \alpha_{\sigma(i)}$, the resulting surface is connected.
Therefore, we can choose a family $\gamma$ of capping paths  for $\Lambda_i$  on $\overline{\Sigma}_i$  such that none of them intersect the curves $\alpha_0, \alpha_{\sigma(1)}, \dots, \alpha_{\sigma(i)}$.
Choose  families of capping paths for $\Lambda_0, \dots, \Lambda_n$ in a similar way.
As a result, for any rigid holomorphic disk $u$ used in differentials of DGAs and DGA maps,
we only need to count the intersections of curves in $\alpha$ with the disk boundary, i.e. $w(\tilde{u})=w(u)$.

With this choice of capping paths, we can write down the DGA $\big(\calA(\Lambda_{i}; \bbZ_2[H_1(\overline{\Sigma}_{i})]), \partial_i\big)$ combinatorially, for $i=1,\dots, n$.
There are $2i+1$ points on $\Lambda_i$ given by the intersection of $\alpha_0$ and $\Lambda_i$, labeled by $\tilde{s}_0$, along with
the two intersections of $\alpha_{\sigma(j)}$ and $\Lambda_i$, labeled by $\tilde{s}_{\sigma(j)}$ (positive intersection) and $\tilde{s}'_{\sigma(j)}$ (negative intersection), for $j=1,\dots,i$.
One then takes the DGA of $\Lambda_i$ with these $2i+1$ base points, which has coefficients
$\bbZ_2[\tilde{s}^{\pm 1}_0, \tilde{s}_{\sigma(1)}^{\pm 1},{\tilde{s}_{\sigma(1)}}^{' \pm 1},\dots, \tilde{s}^{\pm 1}_{\sigma(i)},{\tilde{s}_{\sigma(i)}}^{' \pm 1}]$,
and quotients by the relations 
$\tilde{s}'_{\sigma(j)}=\tilde{s}_{\sigma(j)}^{-1}$ for $j=1, \dots, i$,
to get the DGA $\big(\calA(\Lambda_{i}; \bbZ_2[H_1(\overline{\Sigma}_{i})]), \partial_i\big),$
which a DGA is over $\bbZ_2[\tilde{s}^{\pm 1}_0, \tilde{s}_{\sigma(1)}^{\pm 1},\dots, \tilde{s}^{\pm 1}_{\sigma(i)}]$ and $\{\tilde{s}_0, \tilde{s}_{\sigma(1)}, \dots, \tilde{s}_{\sigma(i)}\}$ is a basis of $H_1(\overline{\Sigma}_i)$ that correspond to the curves $\alpha_0, \alpha_{\sigma(1)}, \dots, \alpha_{\sigma(i)}$.

Now we are ready to describe the DGA map $\Phi_i$ induced by the exact Lagrangian cobordism $\Sigma_i$, for $i=1,\dots, n$, which corresponds to a pinch move at  crossing $b_{\sigma(i)}$.
When we combine \cite[Section 6.5]{EHK}  with Proposition \ref{map}, we find that the DGA map 
$$
\Phi_i:\big(\calA(\Lambda_{i-1}; \bbZ_2[H_1(\overline{\Sigma}_{i-1})]), \partial_{i-1}\big) \to\big(\calA(\Lambda_{i}; \bbZ_2[H_1(\overline{\Sigma}_{i})]), \partial_i\big)
$$
maps the Reeb chord $b_{\sigma(i)}$ to $ \tilde{s}_{\sigma(i)}$ and any other Reeb chord $c$ to
$$ c+ \displaystyle{\sum_{\dim\calM(c,b_{\sigma(i)};c_1, \dots, c_m)=1}\sum_{u \in {\calM}(c,b_{\sigma(i)}; c_1, \dots, c_m)} w(u)  \tilde{s}^{-1}_{\sigma(i)} \ c_1 \cdots c_m},
$$
where $\calM(c,b_{\sigma(i)};c_1, \dots , c_m)$ is the moduli space of holomorphic disks in $\bbR^2_{xy}$ with boundary on $\Pi_{xy}(\Lambda_{i-1})$ that covers a positive quadrant around  $c$ and $ b_{\sigma(i)}$ and a negative quadrant around $ c_1, \dots, c_m$. 
See \cite[Section 6.5]{EHK} for a detailed definition.

In our case, in order to describe $\Phi_i$ combinatorially, we introduce two notations first.
\begin{defn}
Let $\sigma$ be a permutation of $\{1, \dots, n\}$. For $i \in \{1, \dots, n\}$, we define 
$$
\begin{array}{rl}
T_{\sigma}^i:=&\{j\in\{1,\dots , n \}\mid \sigma^{-1}(j) >\sigma^{-1}(i) \textrm{ and if } 
i<k<j \textrm{ or } j<k<i, \textrm{ then }
\sigma^{-1}(k) < \sigma^{-1}(i) \}\vspace{0.1in},\\

S_{\sigma}^i:=&\{j\in\{1,\dots , n \}\mid i \in T_{\sigma}^j\}\vspace{0.05in}\\
=&
\{j\in\{1,\dots , n \}\mid \sigma^{-1}(j) <\sigma^{-1}(i) \textrm{ and if } 
i<k<j \textrm{ or } j<k<i, \textrm{ then }
\sigma^{-1}(k) < \sigma^{-1}(j) \}.
\end{array}$$
\end{defn}

If
 $j \in T_{\sigma}^{\sigma(i)}$ (an example is shown in Figure \ref{comb}), the map $\Phi_i$  sends $b_j$ to
$$\Phi_i(b_j) = b_j + \tilde{s}_{\sigma(i)}^{-1}  \prod\limits_{\substack{j<k<\sigma(i) \textrm{ or }\\ \sigma(i)<k<j }} \tilde{s}_k^{-2}.$$
For $a_1$, $a_2$ and the rest of $b_j$'s, the map $\Phi_i$ is identity.

\begin{figure}[Iht]
\labellist
\pinlabel $b_j$ at 22 5
\pinlabel $b_{\sigma(i)}$ at 262 5
\pinlabel $\tilde{s}^{-1}_k$ at 140 5
\pinlabel $\tilde{s}_k$ at 140 62
\pinlabel $+$ at  35 36
\pinlabel $+$ at 247 36
\endlabellist
\includegraphics[width=2.5in]{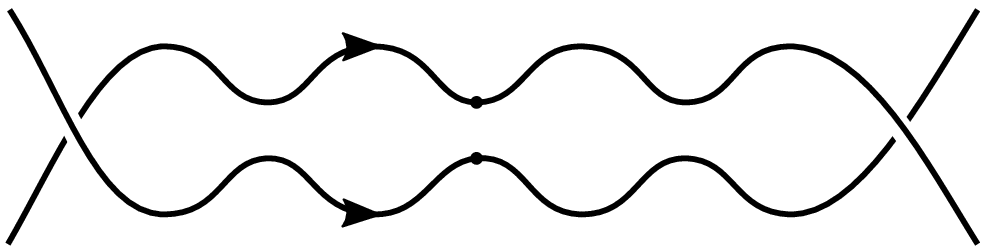}
\caption{A part of the Lagrangian projection of $\Lambda_{i-1}$.}
\label{comb}
\end{figure}

Composing all the maps $\Phi_i$ for  $i=1, \dots, n$ together, we get a DGA map
$$\overline{\Phi}_n : \big(\calA(\Lambda; \bbZ_2[H_1(\Lambda)]), \partial\big) \to\big(\calA(\Lambda_{n}; \bbZ_2[H_1(\overline{\Sigma}_{n})]), \partial_n\big).
$$
that is identity on the Reeb chords $a_1$, $a_2$  and sends the Reeb chord $b_i$, for $i=1, \dots, n$ to
$$
\displaystyle{\overline{\Phi}_n(b_i) = \Phi_1 \circ \cdots \circ \Phi_{\sigma^{-1}(i)}(b_i)= \tilde{s}_i + \sum\limits_{ j \in S_{\sigma}^i}
\Bigg(
\tilde{s}_j^{-1}\prod\limits_{\substack{j<k<i \textrm{ or }\\ i<k<j }} \tilde{s}^{-2}_k \Bigg)}.
$$

Now we describe the last DGA map $$\Phi_{n+1}: \big(\calA(\Lambda_{n}; \bbZ_2[H_1(\overline{\Sigma}_{n})]), \partial_n\big) \to \big(\bbZ_2[H_1(L_{\sigma})], 0\big).$$
As shown in Figure \ref{sn}, the underlying algebra  of $\Lambda_n$ is generated by $a_1$ and $a_2$ and the differential is given by
$$
\begin{array}{rll}
\partial_n(a_1)  &=& \tilde{s}_1 \tilde{s}_2  \  \cdots \ \tilde{s}_n  + \tilde{s}^{-1}_0,\vspace{0.05in}\\
\partial_n(a_2) &=& \tilde{s}_n \tilde{s}_{n-1}\  \cdots \ \tilde{s}_1 + 1.\\
\end{array}$$

\begin{figure}[!ht]
\labellist
\small
\pinlabel $\tilde{s}_1$ at 18 103
\pinlabel $\tilde{s}^{-1}_1$ at 20 52
\pinlabel $\tilde{s}_2$ at 58 103
\pinlabel $\tilde{s}^{-1}_2$ at 57 52
\pinlabel $\tilde{s}_3$ at 95 103
\pinlabel $\tilde{s}^{-1}_3$ at 95 52
\pinlabel $\tilde{s}_{n}$ at 230 103
\pinlabel $\tilde{s}^{-1}_{n}$ at 233 52
\pinlabel $\tilde{s}_0$ at 285 133
\pinlabel $a_1$ at 260 122
\pinlabel $a_2$ at 261 37
\endlabellist
\includegraphics[width=3in]{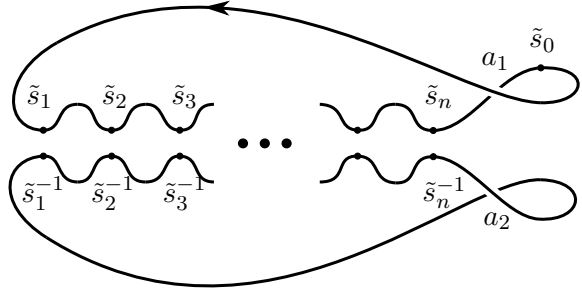}
\caption{The Lagrangian projection of $\Lambda_n$.}
\label{sn}
\end{figure}

Consider the map $\psi: H_1(\overline{\Sigma}_n) \to H_1(L_\sigma)$ induced by the inclusion map $\overline{\Sigma}_n \hookrightarrow L_\sigma$.
Since  the DGA map  
$$\Phi_{n+1}: \big(\calA(\Lambda_{n}; \bbZ_2[H_1(\overline{\Sigma}_{n})]), \partial_n\big) \to \big(\bbZ_2[H_1(L_{\sigma})], 0\big)$$
satisfies that $ \Phi_{n+1} \circ \partial_n = 0 \circ \Phi_{n+1} =0$,
we have $\psi(\tilde{s}_0)=1$ and $\psi(\tilde{s}_1)\psi(\tilde{s}_2) \cdots \psi(\tilde{s}_n)=1$.
Given that the map $\psi$ is surjective,
we assume a basis of $H_1(L_\sigma)$ is $\{s_1, \dots, s_{n-1}\}$, where $s_i = \tilde{s}_i$, for $i=1, \dots, n-1$.
The DGA map $\Phi_{n+1}$ is given by 
$$
\begin{array}{rll}
 a_1 & \mapsto & 0, \vspace{0.05in}\\

a_2& \mapsto & 0, \vspace{0.05in}\\

\tilde{s}_0& \mapsto &1,\vspace{0.05in}\\

\tilde{s}_i & \mapsto & s_i, \hspace{0.1in} i=1, \dots, n-1,\vspace{0.05in}\\

 \tilde{s}_n & \mapsto & (s_1s_2  \cdots s_{n-1})^{-1}.
 \end{array}
 $$
 Composing $\Phi_{n+1}$ with $\overline{\Phi}_n$, we get the augmentation $\epsilon_{\sigma}$ induced by $L_{\sigma}$ as follows.
\begin{thm}\label{augcompute}
Given a permutation $\sigma$ of  $\{1, \dots, n\}$, let $L_{\sigma}$ be  the exact Lagrangian filling of
the Legendrian  $(2,n)$ torus knot $\Lambda$  constructed from the \cite{EHK} algorithm.
If we write 
$$
\begin{array}{rll}
\bbZ_2[H_1(\Lambda)]&=& \bbZ_2[\tilde{s}_0, \tilde{s}_0^{-1}],\vspace{0.05in}\\
\bbZ_2[H_1(L_\sigma)]&=& \bbZ_2[s_1^{\pm 1}, \dots, s_{n-1}^{\pm 1}],
\end{array}
$$
and  set $s_n=(s_1s_2\cdots s_{n-1})^{-1}$,
then 
 the augmentation 
 $$\epsilon_{\sigma} : \calA(\Lambda; \bbZ_2[H_1(\Lambda)]) \to \bbZ_2[H_1(L_{\sigma})] 
$$
induced by $L_{\sigma}$  is given by
$$
\begin{array}{rl}
\epsilon_{\sigma}(a_j)&=0, \hspace{0.2in}j=1,2;\vspace{0.05in}\\
\epsilon_{\sigma}(b_i) &= s_i + \displaystyle{\sum\limits_{ j \in S_{\sigma}^i}
\Bigg(
s_j^{-1}\prod\limits_{\substack{j<k<i \textrm{ or }\\ i<k<j }} s^{-2}_k \Bigg)},
 \hspace{0.2in} i=1,\dots, n; \vspace{0.05in}\\
\epsilon_{\sigma}(\tilde{s}_0)&=1.\\
\end{array}
$$
\end{thm}

\begin{eg}
As an example, we compute the augmentation $\epsilon_{(2,3,1)}$ of the Legendrian $(2,3)$ torus knot induced by the exact Lagrangian filling $L_{(2,3,1)}$.

\begin{figure}[!ht]
\labellist
\small
{
\pinlabel  $b_1$ at 30 517
\pinlabel  $b_1$ at 30 382
\pinlabel  $b_1$ at 30 245
\pinlabel  $b_2$ at 90 517
\pinlabel  $b_3$ at 145 517
\pinlabel  $b_3$ at 145 382
\pinlabel $a_1$ at  195 545
\pinlabel $a_2$ at  188 493
\pinlabel $a_1$ at  195 410
\pinlabel $a_2$ at  188 358
\pinlabel $a_1$ at  195 274
\pinlabel $a_2$ at  190 223
\pinlabel $\tilde{s}_2$ at 85 380
\pinlabel $\tilde{s}^{-1}_2$ at 88 425
\pinlabel $\tilde{s}_2$ at 85 243
\pinlabel $\tilde{s}^{-1}_2$ at 88 288
\pinlabel $\tilde{s}_2$ at 85 107
\pinlabel $\tilde{s}^{-1}_2$ at 88 153
\pinlabel $\tilde{s}_3$ at 140 243
\pinlabel $\tilde{s}^{-1}_3$ at 142 288
\pinlabel $\tilde{s}_1$ at 30 107
\pinlabel $\tilde{s}^{-1}_1$ at 35 153
\pinlabel $\tilde{s}_3$ at 140 107
\pinlabel $\tilde{s}^{-1}_3$ at 142 153
\pinlabel $a_1$ at  195 137
\pinlabel $a_2$ at  190 88
\pinlabel $\tilde{s}_0$ at  208 583
\pinlabel $\tilde{s}_0$ at  208 448
\pinlabel $\tilde{s}_0$ at  208 313
\pinlabel $\tilde{s}_0$ at  208 178
}
\pinlabel $b_1$ at 400 550
\pinlabel $b_2$ at 550 550
\pinlabel $b_3$ at 700 550
\pinlabel $b_1+\tilde{s}_2^{-1}$ at 400 420
\pinlabel $\tilde{s}_2$ at 550 420
\pinlabel $b_3+\tilde{s}_2^{-1}$ at 700 420
\pinlabel $b_1+\tilde{s}_2^{-1}+\tilde{s}_2^{-2}\tilde{s}_3^{-1}$ at 400 285
\pinlabel $\tilde{s}_2$ at 550 285
\pinlabel $\tilde{s}_3+\tilde{s}_2^{-1}$ at 700 285
\pinlabel $\tilde{s}_1+\tilde{s}_2^{-1}+\tilde{s}_2^{-2}\tilde{s}_3^{-1}$ at 400 150
\pinlabel $\tilde{s}_2$ at 550 150
\pinlabel $\tilde{s}_3+\tilde{s}_2^{-1}$ at 700 150
\pinlabel $s_1+s_2^{-1}+s_1s_2^{-1}$ at 400 20
\pinlabel $s_2$ at 550 20
\pinlabel $s_1^{-1}s_2^{-1}+s_2^{-1}$ at 700 20
\pinlabel $\Phi_1$ at 280 480
\pinlabel $\Phi_2$ at 280 350
\pinlabel $\Phi_3$ at 280 215
\pinlabel $\Phi_4$ at 280 85

{\large
\pinlabel $\downarrow$ at 400 480
\pinlabel $\downarrow$ at 550 480
\pinlabel $\downarrow$ at 700 480
\pinlabel $\downarrow$ at 400 350
\pinlabel $\downarrow$ at 550 350
\pinlabel $\downarrow$ at 700 350
\pinlabel $\downarrow$ at 400 215
\pinlabel $\downarrow$ at 550 215
\pinlabel $\downarrow$ at 700 215
\pinlabel $\downarrow$ at 400 85
\pinlabel $\downarrow$ at 550 85
\pinlabel $\downarrow$ at 700 85
}
\endlabellist
\hspace{-3in}
\includegraphics[width=1.5in]{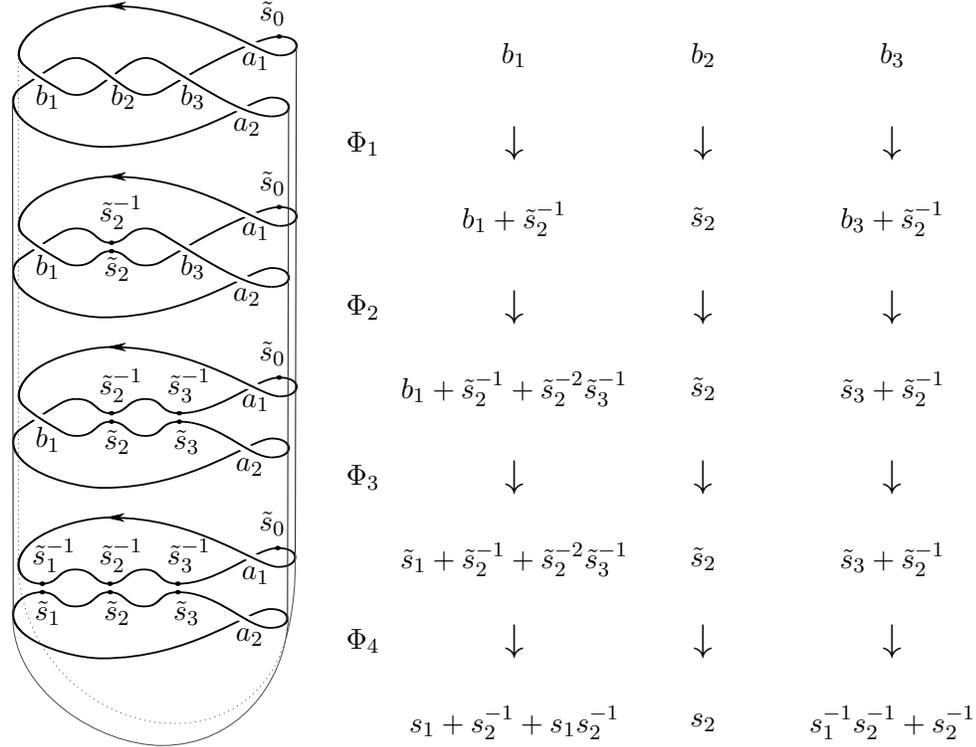}

\vspace{0.1in}

\caption{A computation of the augmentation induced by an exact Lagrangian filling of the Legendrian $(2,3)$ torus knot.
We keep track of the image of $b_1, b_2, b_3$ under the composition of $\Phi_1, \Phi_2, \Phi_3$ and $\Phi_4$.
The last line is the image of $b_1, b_2, b_3$ under the augmentation $\epsilon_{(2,3,1)}$.
}
\end{figure}

Similarly, one can compute the augmentation for each permutation of $\{1,2,3\}$ and get the following result.
$$
\arraycolsep=6pt
\def\arraystretch{1.5}
\begin{array}{ccccc}
\epsilon & &\epsilon(b_1) &\epsilon(b_2)& \epsilon(b_3)\\
\epsilon_{(1,2,3)} & & s_1& s_2+s_1^{-1}& s_1^{-1}s_2^{-1} + s_2^{-1}\\
\epsilon_{(1,3,2)} =\epsilon_{(3,1,2)} && s_1& s_2 +s_1^{-1}+ s_1s_2&s_1^{-1}s_2^{-1}\\
\epsilon_{(2,1,3)} && s_1 + s_2^{-1}&s_2  & s_1^{-1}s_2^{-1}+s_2^{-1}+ s_1^{-1}s_2^{-2}\\
\epsilon_{(2,3,1)}&& s_1 + s_2^{-1}+ s_1s_2^{-1}& s_2  & s_1^{-1}s_2^{-1}+s_2^{-1}\\
\epsilon_{(3,2,1)} &&s_1 + s_2^{-1}&s_2+ s_1s_2& s_1^{-1}s_2^{-1}\\
\end{array}
$$
\end{eg}

\subsection{Proof of the main theorem}\label{equi}
In this section, we use Theorem \ref{augcompute} to find an invariant of augmentations induced from the exact Lagrangian fillings obtained from the \cite{EHK} algorithm.
As a result, we distinguish all the augmentations in Theorem \ref{augcompute} and thus prove Theorem \ref{main}.

\begin{lem}\label{commute}
Let $L_1$ and $L_2$ be two exact Lagrangian fillings of the Legendrian $(2,n)$ torus knot $\Lambda$ constructed from the \cite{EHK} algorithm.
If  $L_1$ and $L_2$ are exact Lagrangian isotopic, then
there exists an invertible map $g: H_1(L_1)\to H_1(L_2)$ such that the following diagram commutes:
\begin{equation}\label{diag}
\xymatrix{(\calA(\Lambda),\partial) \arr^{Id} \ard_{\epsilon_{L_1}}& (\calA(\Lambda),\partial) \ard_{\epsilon_{L_2}}\\
\bbZ_2[H_1(L_1)] \arr_g &\bbZ_2[H_1(L_2)],}
\end{equation}
where $\epsilon_{L_1}$ and $\epsilon_{L_2}$ are augmentations induced by $L_1$ and $L_2$ respectively.
\end{lem}

\begin{proof}
The isotopy between
$L_1$ and $L_2$  induces an invertible map $g: H_1(L_1)\to H_1(L_2).$ 
If we identify both  $H_1(L_1)$ and $H_1(L_2)$ with $\bbZ^{n-1}$, then $g\in GL(n-1,\bbZ)$.
Thus, we have two augmentations of $\calA(\Lambda)$ to $\bbZ_2[H_1(L_2)]$: $\epsilon_1=g\circ \epsilon_{L_1}$ and $\epsilon_2=\epsilon_{L_2}$.
Since the two fillings $L_1$ and $L_2$ are isotopic through a family of exact Lagrangian fillings, according to \cite[Theorem 1.3]{EHK}, we know that $\epsilon_1$ and $\epsilon_2$ are chain homotopic.
In other words, there exists a degree $1$ map
$H: \calA(\Lambda) \to \bbZ_2[H_1(L_2)]$ such that $H\circ \partial = \epsilon_1 -\epsilon_2$
as one can see from following diagram, where $C_i$ denotes degree $i$ part of $\calA(\Lambda)$. 
\diag{&C_{-1} \ar[rd]^{H} \arl&  C_0 \ar[rd]^{H}  \arl_{\partial} \ar@<1.5pt>[d]_{
 \epsilon_1}\ar@<-2pt>[d]^{
\hspace{0.03in}
\epsilon_2} & C_1 \arl_{\partial} & \arl\\
&0 \arl & \bbZ_2[H_1(L_2)] \arl & 0 \arl & \arl}
Note that $\Lambda$ has a Lagrangian projection (as shown in Figure \ref{2nast}) such that no Reeb chords are in negative degree.
Hence $C_{-1}=0$ and  $\epsilon_1-\epsilon_2 = H\circ \partial =0$.
Therefore $\epsilon_1=\epsilon_2$, i.e. the diagram \eqref{diag} commutes.
\end{proof}
\begin{rmk}
For any DGA $\calA$ that vanishes on degree $-1$ part, by the same argument, we have that two augmentations $\epsilon_1$ and $\epsilon_2$ of $\calA$ are chain homotopic if and only if they are identically the same.
For a more general criteria of two augmentations to be chain homotopic, one can check \cite[Proposition 5.16]{NRSSZ}.
\end{rmk}

Therefore, in order to distinguish exact Lagrangian fillings, we only need to distinguish their induced augmentations up to a $GL(n-1, \bbZ)$ action.
Observing the formula of the augmentation $\epsilon_{\sigma}$ in Theorem \ref{augcompute},   
we get a  combinatorial way to
 define the number of terms in $\epsilon_{\sigma}(b_i)$ for $i=1,\dots, n$ as follows.
\begin{defn}
For each permutation $\sigma$ of $\{1, \dots, n\}$ and any number $i \in \{1, \dots, n\}$, we define 
$C_{\sigma}:=\left(C_{\sigma}^1, C_{\sigma}^2, \dots , C_{\sigma}^n \right)$,
where $C_{\sigma}^i =| S_{\sigma}^i|+1$.
\end{defn}

\begin{eg}
We compute the vector $C_{\sigma}$ for all the permutations $\sigma$ of $\{1,2,3\}$ as follows.
$$
\arraycolsep=12pt
\def\arraystretch{1.8}
\begin{array}{|c|c|c|c|c|c|}
\sigma & (1,2,3) & (1,3,2)\sim(3,1,2) & (2,1,3)& (2,3,1)& (3,2,1)\\
\hline
C_{\sigma} & (1,2,2) & (1, 3,1)&( 2,1,3) & (3,1,2) & (2,2,1)\\
\end{array}
$$
\end{eg}

\begin{prop}\label{inva}
If two exact Lagrangian fillings $L_{\sigma_1}$ and $L_{\sigma_2}$ are exact Lagrangian isotopic, 
then $C_{\sigma_1}=C_{\sigma_2}$.
In other words, the vector $C_{\sigma}$ is an invariant of the exact Lagrangian filling $L_{\sigma}$ up to exact Lagrangian isotopy.
\end{prop}
\begin{proof}
As noted from the formula in  Theorem \ref{augcompute} that $C_{\sigma}^i$ is the number of terms in $\epsilon_{\sigma}(b_i)$.
To show that, we first need prove that $\epsilon_{\sigma}(b_i)$ as a sum of monomials can not be shorter, 
i.e, no terms in $\epsilon_{\sigma}(b_i)$ can be canceled by another term.
This can be done by observing each monomial as a polynomial of $s_1,\dots, s_{n-1}$.
If $i \neq n$, then each term of $\epsilon_{\sigma}(b_i)$ is one of the following forms:
\begin{enumerate}
\item $s_i$;
\item $\displaystyle{s^{-1}_k \prod_{j\in S}s^{-2}_j }$ for some $k \neq i \in \{1, \dots, n-1\}$ and a subset  $S \subset \{1, \dots, n-1\}$ that does not contain $i,k$ (can be a empty set);
\item $\displaystyle{{\prod_{j\in T} s_j^{-1}} \prod_{k \notin T} s_k}$ for some subset $T \subset \{1, \dots, n-1\} $ that does not contain $i$ (can be a empty set).
\end{enumerate}
If $i=n$, each term of $\epsilon_{\sigma}(b_n)$ can be either ${s^{-1}_1 \cdots s^{-1}_{n-1}}$ or 
the form $(2)$ above.
Comparing  degrees of $s_1, \dots, s_{n-1}$ of each monomial, we know that no terms can be canceled.

Note that any map $g \in GL(n-1, \bbZ)$ does not change the number of terms of $\epsilon_{\sigma}(b_i)$. 
Therefore, combining with Lemma \ref{commute}, we prove the proposition.
\end{proof}

We say that two permutations $\sigma_1$ and $\sigma_2$ of $\{1, \dots, n\}$ are  {\bf isotopy equivalent} 
if they are equivalent via a sequence of relations of the form
 \begin{equation}\label{relation}
 (\dots,i,j,\dots,k,\dots) \sim (\dots,j,i,\dots,k,\dots) \hspace{0.1in}  \textrm{where } i<k<j.
 \end{equation}
By \cite{EHK}, if $\sigma_1$ and $\sigma_2$ are isotopy equivalent, the corresponding exact Lagrangian fillings $L_{\epsilon_1}$ and $L_{\epsilon_2}$ are exact Lagrangian isotopic and hence $C_{\sigma_1}=C_{\sigma_2}$.
Conversly, we have the following Lemma.

\begin{lem}\label{isotopy}
If $C_{\sigma_1} = C_{\sigma_2}$, then $\sigma_1$ and $\sigma_2$ are isotopy equivalent.
\end{lem}
\begin{proof}
If $\sigma_1(1)=k$, then $C_{\sigma_1}^k=1$. So $C_{\sigma_2}^k=1$, i.e., we have that  $S_{\sigma_2}^k=\emptyset$. 
If  $\sigma_2(1)\neq k$, assume the element in $\sigma_2$ right before $k$  is $l$, i.e. $\sigma_2\big(\sigma_2^{-1}(k)-1\big)=l$.
Note that $l\notin S_{\sigma_2}^k$,
i.e. there exists $i$ such that $l<i<k$ or $k<i<l$ and $\sigma_2^{-1}(i)>\sigma_2^{-1}(l)$.
Note that $i \neq k$ and hence 
  $\sigma_2^{-1}(i)>\sigma_2^{-1}(k)= \sigma_2^{-1}(l)+1$.
Thus we can use  the relation \eqref{relation} to switch  $l$ and $k$.
In this way we can switch $k$ to  the first position in $\sigma_2$, i.e. $\sigma_2(1)=k=\sigma_1(1) $.

By induction, assume $\sigma_2(i)=\sigma_1(i)$ for $i<l$ and $\sigma_1(l)=k$.
We have $S_{\sigma_1}^k \subset S_{\sigma_2}^k$. 
The assumption $C_{\sigma_2}^k=C_{\sigma_1}^k$ implies that $|  S_{\sigma_1}^k|=  |  S_{\sigma_2}^k|$ and hence $S_{\sigma_1}^k=S_{\sigma_2}^k.$ 
If $\sigma_2(l)\neq k$, for a similar reason as above, one can switch $k$  to the $l$-th position and get $\sigma_2(l)=\sigma_1(l)$.
Therefore, we have that $\sigma_1$ and $\sigma_2$ are isotopy equivalent.
\end{proof}
\begin{thm}\label{knot}
If $n$ is odd,
the $C_n$ exact Lagrangian fillings  of the Legendrian $(2,n)$ torus knot $\Lambda$ from the   algorithm in \cite{EHK}  are all of different exact Lagrangian isotopy classes.
\end{thm}
\begin{proof}
If two augmentations $\sigma_1$ and $\sigma_2$ are not isotopy equivalent, by Lemma \ref{isotopy}, we have $C_{\sigma_1}\neq C_{\sigma_2}$.
According to Proposition \ref{inva}, the corresponding exact Lagrangian fillings $L_{\sigma_1}$ and $L_{\sigma_2}$ are not exact Lagrangian isotopic. 
Therefore, the Legendrian $(2,n)$ torus knot has at least $C_n$ exact Lagrangian fillings up to exact Lagrangian isotopy. 
\end{proof}
\begin{cor}\label{link}
When $n$ is even,
the  Legendrian $(2,n)$ torus link $\Lambda$ has at least $C_n$ exact Lagrangian fillings.
\end{cor}
\begin{proof}
Start with the Legendrian $(2,n+1)$-knot  $\Lambda_0$ and label its degree $0$ Reeb chords from left to right by $b_1, \dots, b_{n+1}$ as usual.
Let $\Sigma$ be the exact Lagrangian cobordism from $\Lambda$ to $\Lambda_0$ that corresponds to a pinch move of $\Lambda_0$ at $b_{n+1}$.
For any permutation $\sigma$ of $\{1, \dots, n\}$, the exact Lagrangian filling $L_{\sigma}$ of $\Lambda$ gives an exact Lagrangian filling of $\Lambda_0$ by concatenating with $\Sigma$ on the top.
This new exact Lagrangian filling of $\Lambda_0$ corresponds to the permutation $\tilde{\sigma}=(n+1, \sigma(1), \dots, \sigma(n))$ of $\{1, 2, \dots, n+1\}$, i.e., it is the filling $L_{\tilde{\sigma}}$ of $\Lambda_0$.
Note that $C^{n+1}_{\tilde{\sigma}}=1$.
Moreover, we have that $C^i_{\tilde{\sigma}}= C^i_{\sigma}$ for $i=1, \dots, n-1$ and $C^n_{\tilde{\sigma}}=C^n_{\sigma}+1$.
Thus $C_{\tilde{\sigma}}$ is determined by $C_{\sigma}$.
Therefore, by Proposition \ref{inva} and  Lemma \ref{isotopy},
if two permutations $\sigma_1$ and $\sigma_2$ of $\{1, \dots, n\}$ are not isotopy equivalent,
their induced permutations $\tilde{\sigma}_1$ and $\tilde{\sigma}_2$ of $\{1, \dots , n+1\}$ are not isotopy equivalent.
According to Theorem \ref{knot}, the corresponding exact Lagrangian fillings $L_{\tilde{\sigma}_1}$ and $L_{\tilde{\sigma}_2}$ of $\Lambda_0$ are not exact Lagrangian isotopic.
Hence $L_{\sigma_1}$ and $L_{\sigma_2}$  are not exact Lagrangian isotopic.
\end{proof}

\bibliographystyle{alpha}

\begin{thebibliography}{}

\bibitem[CEL10]{CEL}
K.~Cieliebak, T.~Ekholm, and J.~Latschev.
\newblock Compactness for holomorphic curves with switching {L}agrangian
  boundary conditions.
\newblock {\em J. Symplectic Geom.}, 8(3):267--298, 2010.

\bibitem[Che02]{Che}
Yuri Chekanov.
\newblock Differential algebra of {L}egendrian links.
\newblock {\em Invent. Math.}, 150(3):441--483, 2002.

\bibitem[EENS13]{EENS}
Tobias Ekholm, John~B. Etnyre, Lenhard Ng, and Michael~G. Sullivan.
\newblock Knot contact homology.
\newblock {\em Geom. Topol.}, 17(2):975--1112, 2013.

\bibitem[EGH00]{EGH}
Y.~Eliashberg, A.~Givental, and H.~Hofer.
\newblock Introduction to symplectic field theory.
\newblock Number Special Volume, Part II, pages 560--673. 2000.
\newblock GAFA 2000 (Tel Aviv, 1999).

\bibitem[EHK]{EHK}
Tobias Ekholm, Ko~Honda, and Tam{\'a}s K{\'a}lm{\'a}n.
\newblock {L}egendrian knots and exact {L}agrangian cobordisms.
\newblock {\em J. \ Eur. \ Math. \ Soc. (JEMS)}.
\newblock To appear.

\bibitem[Eli98]{Eli}
Yakov Eliashberg.
\newblock Invariants in contact topology.
\newblock In {\em Proceedings of the {I}nternational {C}ongress of
  {M}athematicians, {V}ol. {II} ({B}erlin, 1998)}, number Extra Vol. II, pages
  327--338, 1998.

\bibitem[ENS02]{ENS}
John~B. Etnyre, Lenhard~L. Ng, and Joshua~M. Sabloff.
\newblock Invariants of {L}egendrian knots and coherent orientations.
\newblock {\em J. Symplectic Geom.}, 1(2):321--367, 2002.

\bibitem[Lin16]{Lin}
Francesco Lin.
\newblock Exact lagrangian caps of {L}egendrian knots.
\newblock {\em J. Symplectic Geom.}, 14(1):269--295, 2016.

\bibitem[Ng03]{Ngresolve}
Lenhard~L. Ng.
\newblock Computable {L}egendrian invariants.
\newblock {\em Topology}, 42(1):55--82, 2003.

\bibitem[Ng10]{NgRSFT}
Lenhard Ng.
\newblock Rational symplectic field theory for {L}egendrian knots.
\newblock {\em Invent. Math.}, 182(3):451--512, 2010.

\bibitem[NRS{\etalchar{+}}15]{NRSSZ}
Lenhard Ng, Dan Rutherford, Steven Sivek, Vivek Shende, and Eric Zaslow.
\newblock Augmentations are sheaves.
\newblock {\em arXiv preprint arXiv:1502.04939}, 2015.

\bibitem[STWZ15]{STWZ}
Vivek Shende, David Treumann, Harold Williams, and Eric Zaslow.
\newblock Cluster varieties from {L}egendrian knots.
\newblock {\em arXiv preprint arXiv:1512.08942}, 2015.

\end{thebibliography}

\newcommand{\etalchar}[1]{$^{#1}$}

\end{document}